\documentclass{amsart}    
\usepackage{amsmath} 
\usepackage{paralist}
\usepackage{graphics} %% add this and next lines if pictures should be in esp format
\usepackage{epsfig} %For pictures: screened artwork should be set up with an 85 or 100 line screen
\usepackage{graphicx}  \usepackage{epstopdf} %This is to transfer .eps figure to .pdf figure; please compile your paper using  PDFLaTex or PDFTeXify.
\usepackage[colorlinks=true]{hyperref}
\hypersetup{urlcolor=blue, citecolor=blue}

\newtheorem{Proposition}{Proposition}[section]

\newtheorem{Lemme}{Lemma}[section]
\newtheorem{Theoreme}{Theorem}[section]

\newtheorem{Corollaire}{Corollary}[section]
\newtheorem{Remarque}{Remark}
\def \vu{\vec{u}}

\def \P{\mathbb{P}}

\def \R{\mathbb{R}}
\def \Rt{\mathbb{R}^3}

\def \ds{\displaystyle}

\newcommand{\norm}[1]{\left \Vert #1 \right \Vert}
\usepackage{amssymb}
\makeatletter
\@namedef{subjclassname@2020}{\textup{2020} Mathematics Subject Classification}
\makeatother
\title[\bf Fractional approach  to quadratic nonlinear parabolic systems] %
{On the  fractional approach  to quadratic nonlinear parabolic systems} 
\author[  Oscar Jarr\'in and Geremy Loacham\'in]{}
\date{\today}

\subjclass[2020]{Primary: 35B30, 35B40; Secondary: 35D35, 35K58}
\keywords{Fractional Laplacian operator, fractional approximation in the strong topology,  mild solutions, Navier-Stokes equations and related models, Keller-Segel system}  

\email{oscar.jarrin@udla.edu.ec}
\email{geremy.loachamin@uni.lu} 

\thanks{$^*$Corresponding author:  Oscar Jarr\'in}

\begin{document}
	\maketitle
	
	% Enter the first author's name and address:
 
	\centerline{\scshape Oscar Jarr\'in$^*$}
	\medskip
	{\footnotesize
		% Enter the address of the first author
		\centerline{Escuela de Ciencias Físicas y Matemáticas}
		\centerline{Universidad de Las Américas}
		\centerline{V\'ia a Nay\'on, C.P.170124, Quito, Ecuador}
	} 

\medskip 

	\centerline{\scshape Geremy Loacham\'in}
	{\footnotesize
	\centerline{Faculty of Science}
	\centerline{University of Luxembourg} 
	\centerline{Maison du Nombre, 6 Avenue de la Fonte, Esch-sur-Alzette L-4364, Luxembourg} 
}
	
\bigskip
%%%%%%%%%%%%%%%%%%%%%%%%%%%%%%%%%%%%%%%%%%%%%%
\begin{abstract} We introduce a general coupled system of parabolic equations with quadratic nonlinear terms and diffusion terms defined by fractional powers of the Laplacian operator. We develop a method to establish the rigorous convergence of the fractional diffusion case to the classical diffusion case in the strong topology of Sobolev spaces, with explicit convergence rates that reveal some unexpected phenomena. 

These results apply to several relevant real-world models included in the general system, such as the Navier-Stokes equations, the Magneto-hydrodynamics equations, the Boussinesq system, and the Keller-Segel system. For these specific models, this fractional approach is further motivated by previous numerical and experimental studies.  
\end{abstract}
%%%%%%%%%%%%%%%%%%%%%%%%%%%%%%%%%%%%%%%%%%%%%%
%%%%%%%%%%%%%%%%%%%%%%%%%%%%%%%%%%%%%%%%%%%%%%
%{\footnotesize \tableofcontents}
%%%%%%%%%%%%%%%%%%%%%%%%%%%%%%%%%%%%%%%%%%%%%%
\section{Introduction}
\subsection{Motivation} In mathematical, physical and biological studies, the fractional Laplacian operator $(-\Delta)^{\frac{\alpha}{2}}$ (for a definition, see expression (\ref{Lap-Frac}) below) is successfully used to describe anomalous diffusion models. These models are essentially expressed as a nonlinear parabolic system of equations, with the fractional Laplacian operator appearing in the diffusion term \cite{Funaki,Jourdain,Metzler,Zaslavsky}.
 %The probabilistic interpretation of evolution problems involving the fractional Laplacian operator was rigorously studied in \cite{Jourdain}, where the authors considered a family of nonlinear integro-differential equations with the fractional Laplacian operator and a nonlocal quadratic term. They then proved a propagation of chaos phenomenon for solutions, which are interpreted as the density of a probability distribution.

\medskip

Mathematically, these systems are also of great interest. On the one hand, in relation to certain relevant models arising from fluid dynamics, we can mention the \emph{incompressible Navier-Stokes equations} (\ref{NS}), which mathematically express the momentum balance for Newtonian and incompressible fluids \cite{PG1}. Additionally, the coupled \emph{Magneto-hydrodynamic equations} (\ref{MHD}) and the \emph{Boussinesq system} (\ref{Keller-Segel}) describe the dynamics of incompressible fluids under the effects of magnetic fields and temperature, respectively \cite{Pedlosky,Shercliff}. The fractional versions of these equations and systems (incorporating the fractional Laplacian operator in the diffusion term) have been employed as significant modifications of the classical equations. These modifications aim to provide a deeper understanding of some outstanding mathematical open questions, such as the global existence and regularity of solutions \cite{Bessaih,Holst,Xu}.

\medskip  

On the other hand, with respect to certain relevant biological models, it is worth mentioning the \emph{Keller-Segel system} (\ref{Keller-Segel}), which models chemotactic aggregation in cellular systems \cite{Keller-Segel}. The fractional version of this system was first proposed in \cite{Escudero} to describe the distribution of population density undergoing random motion governed by a L\'evy process. Mathematical studies of the fractional Keller-Segel system have primarily focused on local and global well-posedness, as well as certain regularity issues \cite{Wu,Zhai}.

\medskip

In the non-exhaustive list of previous works cited above, we emphasize that fractional powers $\alpha$ of the Laplacian operator are employed to more precisely describe how the weak or strong effects of the diffusion term influence the qualitative properties of solutions, particularly their well-posedness and regularity. In this paper, we focus on a different problem concerning fractional diffusion parabolic systems. Specifically, we aim to understand how these systems \emph{approach} their classical diffusion version, which involves the classical Laplacian operator $-\Delta$ in the diffusion term.

\medskip

This problem is particularly interesting, not only from a theoretical perspective, but also in the context of \emph{experimental studies} with fractional Burgers equations \cite{Funaki} and the fractional transport-type equation \cite{Zaslavsky}. These numerical studies reveal that the solutions to fractional equations exhibit behavior similar to that of classical equation (which involve the Laplace operator) when the fractional powers $\alpha$ of the Laplacian operator  are close to the classical value of $2$. 
 
\medskip

From the theoretical point of view, this question has also been explored for certain \emph{elliptic} equations, such as the nonlinear Schrödinger equation \cite{Bieganowski} and the fractional $p$-Laplacian problem \cite{Fernandez-Salort}. In these studies, variational methods and concentration-compactness principles were mainly employed to establish the convergence of \emph{weak solutions} of the fractional problem to those of the classical problem. More precisely, in \cite{Bieganowski}, the authors demonstrated this convergence in the strong topology of the space $L^2_{loc}(\mathbb{R}^d)$ (with $d \geq 3$), while in \cite{Fernandez-Salort}, $\Gamma$-convergence was employed for this purpose.

\medskip

Within the setting of \emph{abstract} semilinear parabolic equations, we would like to mention the recent work \cite{Xu-Ca}, where the authors consider the reaction-diffusion equation:
\[ \partial_t u + (-\Delta)^{\frac{\alpha}{2}} u= f(u)+h. \]
Here, $f=f(t,x,u)$ is a $\mathcal{C}^1$-nonlinear function, and $h$ represents a source term. For each value of the fractional power $0<\alpha<2$, we denote by $u_\alpha : [0,+\infty[\times \R^d \to \R $ the corresponding (global in time) weak solution arising from an initial datum $u_{0,\alpha}\in L^2(\R^d)$, which belongs to the functional space $\mathcal{C}_t L^2_x$. Moreover, we denote by $u_2 \in \mathcal{C}_t L^2_x$ the weak solution of the classical  reaction-diffusion problem
\[ \partial_t u   -\Delta u  = f(u)+h, \]
arising from an initial datum $u_{0,2}$. Then, under a set of required technical conditions on $f$ and $h$, among them the uniform bound
\[  \frac{\partial}{\partial u} f(t,x,u)\leq  \sigma, \]
for all $t\geq0$, $x\in \R^d$, $u\in \R$ and with $\sigma \geq 0$,   one of  the main results of   \cite{Xu-Ca}  (precisely stated in Theorem $4.2$) proves the convergence
\[ u_\alpha(t,x)\to u_2(t,x), \quad \text{in the limit} \ \ \alpha \to 2^{-}, \]
in the weak-$*$ topology of the space $L^\infty([0,T],L^2(\R^d))$, and in  the weak topology of the space $L^2([0,T],L^2(\R^d))$, for any time $T>0$. 

\medskip

The ideas behind the proof of this result are mainly based on the \emph{hemicontinuity} property of the function $f$, along with sharp \emph{a priori} energy estimates, the weak formulation of solutions, and concentration-compactness arguments. See \cite{Rozkosz} for an interesting generalization where the fractional Laplacian operator is replaced by a L\'evy-type integro-differential operator. See also \cite{Bezerra} for a related work concerning the convergence of global attractors of fractional diffusion parabolic systems to the classical diffusion case.

\subsection{Setting} Inspired by these ideas, we study the convergence below within the framework of the following general system of coupled semilinear equations with \emph{quadratic} nonlinearity:
\begin{equation}\label{System}
\begin{cases}
\ds{\partial_t u_i + (-\Delta)^\frac{\alpha_i}{2} u_i + \sum_{j=1}^{n} \sum_{k=1}^{n} Q_{i,j,k}(u_j  u_k)+\sum_{j=1}^{n} L_{i,j}(u_j)=0},\\
u_i(0,\cdot)=u_{0,\alpha_i},
\end{cases}
 \qquad \alpha_i>0, \quad  i=1,\cdots, n.
\end{equation}
Describing each term above, we find that:
\begin{itemize}
	\item The function $\ds{u_i: [0,+\infty] \times \R^d \to \R}$ denotes the $i$-th unknown of the system. Additionally, the function $u_{0,\alpha_i}: \R^d \to \R$ represents the $i$-th initial datum.
	
	\medskip
	
	\item For each parameter $\alpha_i>0$, the fractional Laplacian operator $(-\Delta)^{\frac{\alpha_i}{2}}$, a homogeneous pseudo-differential operator of order $\alpha_i$, is defined in the Fourier variable by
	\begin{equation}\label{Lap-Frac}
	\mathcal{F}\Big( (-\Delta)^{\frac{\alpha_i}{2}} \varphi \Big)(\xi) = C_{\alpha_i, d}\, |\xi|^{\alpha_i}  \mathcal{F}(\varphi)(\xi), \quad \varphi \in \mathcal{S}(\R^d),
	\end{equation}
	where $\mathcal{F}(\cdot)$ denotes the Fourier transform in the spatial variable, and $C_{\alpha_i, d} > 0$ is a constant depending on the parameter $\alpha_i$ and the dimension $d$.

	 \medskip
	 
	\item  The term $Q_{i,j,k}(\cdot)$ is defined by a homogeneous pseudo-differential operator of order $1$. Specifically, at the Fourier level, we consider
	\begin{equation}\label{Term-Q}
	\mathcal{F}\Big( Q_{i,j,k}( \varphi ) \Big) (\xi) = \widehat{q}_{i,j,k}(\xi) \, \mathcal{F}\big(\varphi \big) (\xi), \quad \varphi \in \mathcal{S}(\R^d),
	\end{equation}
	where $\widehat{q}_{i,j,k}$ also denotes the Fourier transform of $q_{i,j,k}$, and  the symbol $\widehat{q}_{i,j,k}: \R^d \setminus \{0\} \to \mathbb{C}$ is a smooth, homogeneous function of order one:  for any $\xi \neq 0$ and any $\lambda > 0$, one has $\widehat{q}_{i,j,k}(\lambda \xi) = \lambda \widehat{q}_{i,j,k}(\xi)$.
	\begin{Remarque}
The \emph{quadratic} nature of the nonlinear term $Q_{i,j,k}(u_ju_k)$ in the system (\ref{System}) is the main difference when compared to the parabolic system $\partial_t u + (-\Delta)^{\frac{\alpha}{2}} u = f(u) + h$ introduced above. In particular, the required assumption on the function $f$, $\frac{\partial}{\partial u} f(t, x, u) \leq \sigma$, restricts the consideration of nonlinear quadratic expressions.
	\end{Remarque}	

\medskip

\item Finally, the term $L_{i,j}(\cdot)$ is defined by a linear and homogeneous pseudo-differential operator of order $0$. In the Fourier variable, we have
\begin{equation}\label{Term-L}
\mathcal{F}\Big( L_{i,j} (\varphi)\Big)(\xi) = \widehat{\ell}_{i,j}(\xi) \mathcal{F}(\varphi)(\xi), \quad \varphi \in \mathcal{S}(\R^d),
\end{equation}
where the symbol $\widehat{\ell}_{i,j}: \R^d \setminus \{0\} \to \mathbb{C}$ is a bounded and smooth function satisfying $\widehat{\ell}_{i,j}(\lambda \xi) = \widehat{\ell}_{i,j}(\xi)$ for any $\xi \neq 0$ and $\lambda > 0$.
\end{itemize}	
 
\medskip 

A simplified version of the system (\ref{System}) was introduced in \cite{PG1} to provide a general framework for studying global-in-time mild solutions. As mentioned, in this work, we examine a completely different problem for the system (\ref{System}). Let $(u_{\alpha_1}, \dots, u_{\alpha_n})$ be a solution to this system. We are interested in studying the asymptotic behavior of these solutions as the parameters $\alpha_i$ approach $2$. When setting $\alpha_i = 2$ in the system (\ref{System}), we formally recover the semi-linear system
\begin{equation}\label{System-Classical}
\begin{cases}
\ds{\partial_t u_i - \Delta u_i + \sum_{j=1}^{n}\sum_{k=1}^{n} Q_{i,j,k}(u_j u_k) + \sum_{j=1}^{n} L_{i,j}(u_j) = 0},\\
u_i(0,\cdot)=u_{0,2,i},
\end{cases}
\end{equation}
which involves the classical Laplacian operator in each diffusion term. We then denote its solution as the vector field $(u_{2,1}, \dots, u_{2,n})$.
Thus, we aim to rigorously derive the convergence
\begin{equation}\label{Convergence-Intro-System}
u_{\alpha_i}(t,x) \to u_{2,i}(t,x), \quad \text{in the limit} \quad \alpha_i \to 2.
\end{equation}

In contrast to previous related works \cite{Bezerra,Rozkosz,Xu-Ca}, we introduce new ideas to study this convergence. Specifically, we exploit the explicit structure of mild solutions to the system (\ref{System}), as defined in expression (\ref{System-Integral}) below. This approach, combined with sharp estimates at the Fourier level for the convolution kernels appearing in the mild formulation, enables us to analyze the convergence of (\ref{Convergence-Intro-System}) in the strong topology of Sobolev spaces. Furthermore, this methodology allows us to derive explicit convergence rates for the solutions, which depend essentially on the prescribed convergence rates of the fractional powers $\alpha_i$ approaching the limit value of $2$. 

\subsection{Some related models}\label{Sec:Models} To close this introductory section, it is worth mentioning that the system (\ref{System}) and its classical diffusion version (\ref{System-Classical}) are of particular interest, as they encompass several relevant real-world models in fluid dynamics and biology. In this context, the study of the convergence (\ref{Convergence-Intro-System}) conducted in this paper systematically applies to the following equations and systems:
\begin{itemize}
	\item When $n=3$ and  $\widehat{\ell_{i,j}}(\xi)\equiv 0$ in (\ref{Term-L}), for appropriately chosen symbols $\widehat{q}_{i,j,k}(\xi)$ in expression (\ref{Term-Q}) to recover the operator $\P \text{div}(\cdot)$, where $\P$ denotes the well-known Leray's projector, one obtains the  \emph{incompressible Navier-Stokes equations}:
	\begin{equation}\label{NS}
	\partial \vu - \Delta \vu + \P \text{div}(\vu \otimes \vu)=0, \quad t\geq 0, \quad  x \in \Rt.
	\end{equation}
Here, $\vu=(u_1, u_2, u_3)$ represents the  velocity of the fluid, which is assumed to be a  divergence-free vector field.  
	
	\medskip
	
	\item Similarly, when $n=6$ and  $\widehat{\ell_{i,j}}(\xi)\equiv 0$ in (\ref{Term-L}), for a divergence-free velocity $\vu=(u_1,u_2,u_3)$ and  a divergence-free magnetic field $\vec{b}=(u_4, u_5, u_6)$, by selecting the same  symbols $\widehat{q}_{i,j,k}(\xi)$ as above in expression (\ref{Term-Q}), one can  recover the \emph{Magneto-hydrodynamic} equations:
	\begin{equation}\label{MHD}
	\begin{cases}\vspace{2mm}
	\partial_t \vu - \Delta \vu + \P \text{div}(\vu \otimes \vu)-\P \text{div}(\vec{b} \otimes \vec{b})=0,\\
	\partial_t \vec{b} - \Delta \vec{b} + \P \text{div}(\vec{b} \otimes \vu)-\P \text{div}(\vu \otimes \vec{b})=0,
	\end{cases} \quad  t\geq 0, \quad  x \in \Rt.
	\end{equation}
	
	\medskip

	\item  On the other hand, in the case of $n=4$, with appropriate  symbols  $\widehat{\ell_{i,j}}(\xi)$ in expression (\ref{Term-L}) to recover  the Leray's projector $\P$, and suitable symbols  $\widehat{q}_{i,j,k}(\xi)$ in (\ref{Term-Q}), for the fixed vector $\vec{e}_3=(0,0,1)$,   a divergence-free velocity $\vu=(u_1,u_2, u_3)$, and a scalar temperature $\theta=u_4$, we obtain the Boussinesq system:
	\begin{equation}\label{Boussinesq}
	\begin{cases}\vspace{2mm}
		\partial_t \vu - \Delta \vu + \P \text{div}(\vu \otimes \vu) - \P(\theta \vec{e}_3)=0, \\
		\partial_t \theta - \Delta \theta + \text{div}(\theta \vu)=0,
	\end{cases} \quad  t\geq 0, \quad  x \in \Rt.
	\end{equation}
	
	\item Finally, recall that the \emph{parabolic-elliptic Keller-Segel system} is written as:
	\begin{equation*}
	\begin{cases}
	\partial_t u - \Delta u + \text{div} (u \vec{\nabla} \phi)=0, \\
	-\Delta \phi =u, 
	\end{cases}  \quad t\geq 0, \quad x \in \R^d \ \ \text{with} \ \ d \geq 2,
	\end{equation*}
	where the scalar function $u$ denotes the density of microorganisms, which drift along the gradient of the chemoattractant density $\phi$.
	  Defining the vector field 
	  \[\vu=(u_1,\cdots,u_d)=\vec{\nabla} \phi=- \frac{1}{\Delta} \vec{\nabla} u, \]
	 we get the parabolic system
	\begin{equation}\label{Keller-Segel}
	\partial_t \vu - \Delta \vu +\sum_{i=1}^{d} \frac{1}{-\Delta} \vec{\nabla} \text{div}  \partial_i (u_i \vu) + \frac{1}{2} \sum_{i=1}^{d} \vec{\nabla}( u^2_i)=0,
	\end{equation}
	which follows from (\ref{System-Classical}) in the case $n=d$, with  $\widehat{\ell_{i,j}}(\xi)\equiv 0$ in  (\ref{Term-L}) and suitable symbols  $\widehat{q}_{i,j,k}(\xi)$  in (\ref{Term-Q}). 
	\end{itemize}	
This list of real-world models is not exhaustive. For instance, the system (\ref{System}) also includes as a particular case some quasi-geostrophic type equations driven by non-local, divergence-free velocity fields \cite{Chamorro}.

 \subsection{Statement of the results}\label{Sec:Main-Results}  Recall that mild solutions to the system (\ref{System}) solve the following (equivalent) coupled system of integral equations:
\begin{equation}\label{System-Integral}
\begin{cases}
\begin{split}
u_i (t,\cdot)= & \,  h_{\alpha_i}(t,\cdot)\ast u_{0,\alpha_i} +  \sum_{j=1}^{n}\sum_{k=1}^{n}\int_{0}^{t}h_{\alpha_i}(t-\tau,\cdot)\ast Q_{i,j,k}(u_j u_k)(\tau,\cdot) d\tau\\
&+\, \sum_{j=1}^{n} \int_{0}^{t} h_{\alpha_i}(t-\tau,\cdot)\ast L_{i,j}(u_j)(\tau,\cdot)d\tau,
\end{split} \\
i=1,\cdots,n.
\end{cases}
\end{equation}
In this system, when $\alpha_i\neq 2$,  the convolution kernel $h_{\alpha_i}(t,\cdot)$ is a fundamental solution of the  fractional heat equation 
\[ \partial_t h_{\alpha_i} + (-\Delta)^{\frac{\alpha_i}{2}} h_{\alpha_i}=0. \]
Similarly, when $\alpha_i=2$, the function $h_2(t,\cdot)$ denotes the well-known heat kernel.

\medskip

To establish our notation, we consider  the vector $(\alpha_1, \cdots, \alpha_n)$,  which contains all the parameters $\alpha_i$. We then  denote by $(u_{0,\alpha_1}, \cdots, u_{0,\alpha_n})$ the vector field of initial data. Consequently, the vector $(u_{\alpha_1}, \cdots, u_{\alpha_n})$ represents the corresponding solutions to the integral system (\ref{System-Integral}).

\medskip

The existence of local-in-time solutions $(u_{\alpha_1}, \cdots, u_{\alpha_n})$ follows from a standard contraction argument in the space $\mathcal{C}([0,T], H^s(\R^d))$, with $T>0$ sufficiently small. In this context, the main purpose of the following proposition is to provide an explicit expression for the existence time $T$, which depends on  $\alpha_i$ and the initial data $u_{0,\alpha_i}$. For simplicity, we will denote $\alpha=(\alpha_1, \cdots, \alpha_n)$ the vector containing all the powers  of the Laplacian operator.

\begin{Proposition}\label{Prop:LWP} For  $s>d/2$,  let  $(u_{0,\alpha_1}, \cdots, u_{0,\alpha_n}) \subset  H^s(\R^d)$ be the initial data. For $i=1,\cdots,n$,  assume that $\alpha_i>1$.   Then,   there exists a time 
\begin{equation}\label{T-alpha}
		0<T_{\alpha}  = \dfrac{1}{2}\,   \min \left[\dfrac{1}{3nC},\ \min_{i,j=1,\cdots,n} \ds{\left( \frac{1-\frac{1}{\alpha_i} }{9n^2  C   \| u_{0,\alpha_j} \|_{H^{s}}  }\right) ^{\frac{\alpha_i}{\alpha_i-1}} }\right],
		\end{equation}
 where $C > 0$ is a generic constant, in addition, there exists  a  function $\ds{u_{\alpha_i} \in  \mathcal{C}\big([0, T_{\alpha}], H^{s}(\R^d)\big)}$,  such that $(u_{\alpha_1}, \cdots, u_{\alpha_n})$ is the unique mild solution   to the coupled integral  system (\ref{System-Integral}).
\end{Proposition}	
\begin{Remarque} In  expression  (\ref{T-alpha}), note that $T_{\alpha} >0$ as long as $\alpha_i>1$. 
\end{Remarque}	

We now present our main result. For clarity, we outline the context and the set of assumptions underlying this result. Within the framework of this proposition:
\begin{itemize}
	\item  In the fractional diffusion case (when $\alpha_i \neq 2$), let $\ds{\big\{  (u_{0,\alpha_1}, \cdots, u_{0,\alpha_n}): \  \alpha_i \neq  2, \ i=1,\cdots,n \big\}}$ denote a family of initial data, and let $\big\{ (u_{\alpha_1},\cdots,u_{\alpha_n}): \   \alpha_i \neq  2, \ i=1,\cdots,n \big\}$ represent the corresponding family of solutions to (\ref{System-Integral}).
	\item  Similarly, in the classical diffusion case (when $\alpha_i=2$), we consider initial data $(u_{0,2,1},\cdots u_{0,2,n})$ and its corresponding solution $(u_{2,1}, \cdots, u_{2,n})$.
\end{itemize}	
We assume the convergence of the initial data in the strong topology of the space $H^s(\R^d)$. Specifically, this convergence is characterized by \emph{prescribed rates}, given by
\begin{equation}\label{Convergence-Data}
\| u_{0,\alpha_i} - u_{0,2,i}\|_{H^s} \leq {\bf c} |2-\alpha_i|^{\beta_i}, \qquad i=1,\cdots,n,
\end{equation}
where ${\bf c}>0$ is a generic constant, and $\beta_i>0$ quantifies the rate at which the fractional power $\alpha_i$ approaches the limit value of $2$.    We therefore investigate whether these convergence rates are preserved for the family of solutions. To this end,   the following remarks are in order. 

\medskip

Note that in the fractional diffusion case, each component  $u_{\alpha_i}$ of the solution is defined at least over the time interval $[0,T_{\alpha}]$, whereas in the classical diffusion case, each component $u_{2,i}$ exists over the interval $[0,T_{2}]$, where by expression (\ref{T-alpha}) one has
\begin{equation}\label{T_2}
0<T_{2}  = \dfrac{1}{2}\,   \min \left[\dfrac{1}{3nC},\ \min_{j=1,\cdots,n} \ds{\left( \frac{ 1 }{18n^2  C   \| u_{0,2,j} \|_{H^{s}}  }\right) ^{2}}\right].
\end{equation}
 By our assumption (\ref{Convergence-Data}) and the explicit expression for the existence times given in (\ref{T-alpha}), a simple computation shows the convergence
\[ T_{\alpha}\to T_{2}, \quad \text{in the limit} \quad \alpha_i \to 2. \]
It is noteworthy that the existence times in the fractional diffusion case also approach those in the classical case. With a minor loss of generality, we shall assume that 
\begin{equation}\label{Assumption-Times}
T_{\alpha} \leq T_2,
\end{equation}
which seems to be the most interesting case when the existence times grow up to their limit. The other possible cases involve minor technical adjustments, which do not affect the main result.

\medskip

We aim to analyze the asymptotic behavior of the solutions $\big\{ (u_{\alpha_1},\cdots,\alpha_n): \   \alpha_i \neq  2, \ i=1,\cdots,n  \big\}$  as each fractional power $\alpha_i$   approaches its limiting value of $2$. Therefore, we assume that each $\alpha_i$ lies within a neighborhood of $2$. Specifically, for a small technical parameter $\delta$, we impose the condition:
\begin{equation}\label{Condition-alpha}
2-\delta < \alpha_i  < 2+\delta, \qquad i=1,\cdots,n, \quad \text{with} \quad 0<\delta < \frac{1}{6}.
\end{equation}
This restriction ensures that the $\alpha_i$ remain sufficiently close to the limiting value, with this closeness measured by $\delta$. Specifically, we consider  the supercritical case,  where $2-\delta <\alpha_i<2$, and the subcritical case, where $2<\alpha_i<2+\delta$. Furthermore, considering this parameter $\delta$ leads us to define the following additional technical quantity:
\begin{equation}\label{Quantities}
\eta:= \frac{1+4\delta}{4-2\delta}, 
\end{equation}
which will appear in our estimates.  Note that, since $0<\delta < \frac{1}{6}$, it follows that  $\frac{1}{4}<\eta <\frac{1}{2}$.

\medskip

Finally, in order to quantify the convergence rate of solutions, for each parameter $\beta_i>0$ given in (\ref{Convergence-Data}), we introduce the following function
\begin{equation}\label{Function}
F_i (z):=\max\left(z, z^{\beta_i}\right), \quad z \geq 0, \quad i=1,\cdots,n. 
\end{equation} 

\medskip

Once we have set the main assumptions and notation,  our main result is the following:
\begin{Theoreme}\label{Th-Main}  Assume (\ref{Convergence-Data}), (\ref{Assumption-Times}) and (\ref{Condition-alpha}).   There exists a constant ${\bf C}>0$, which essentially depends on the initial data $(u_{0,2,1}, \cdots, u_{0,2,n})$, the parameter $\delta$ and the quantity $\eta$ defined in (\ref{Quantities}),  such that  the following estimate holds:
	\begin{equation}\label{Conv-Rate-Sol}
	\sup_{0\leq t \leq T_2} \, \sum_{i=1}^{n}\,   t^\eta \, \| u_{\alpha_i}(t,\cdot) - u_{2,i}(t,\cdot)\|_{H^s} \leq {\bf C}\Big(1+T^{\eta+1}_{2}\Big) \, \max_{i=1,\cdots,n} F_i(|2-\alpha_i|),
	\end{equation}
	with the time $T_2$ given in expression (\ref{T_2}), and with the function $F_i(\cdot)$ defined in expression (\ref{Function}).
\end{Theoreme}

The following observations are noteworthy. In estimate (\ref{Conv-Rate-Sol}), we compare the distance between solutions in the fractional and classical diffusion cases with the maximum of the expressions  $F_i(2-\alpha_i)$. These convergence rates reveal a curious behavior. Specifically, using expression (\ref{Function}), we can explicitly write
\[ F_i(|2-\alpha|)= \max\left( |2-\alpha_i|, |2-\alpha_i|^{\beta_i} \right), \]
where, by assumption (\ref{Condition-alpha}), one has $|2-\alpha_i|<\delta <1$. Thus, we derive the following scenarios based on the parameter $\beta_i>0$ introduced in (\ref{Convergence-Data}). 

\medskip

On the one hand, for small values of $\beta_i$, we obtain that $F_i(|2-\alpha_i|)\sim |2-\alpha_i|^{\beta_i}$. Thus,  as somewhat expected, we  conclude that the convergence rate of solutions is determined by the prescribed convergence rate assumed on the initial data. On the other hand, it is interesting to note that  for large values of $\beta_i$, one has $F_i(|2-\alpha_i|)\sim |2-\alpha_i|$.    Consequently, it follows that a stronger prescribed convergence rate of the data is not reflected in the solutions.

\medskip

To explain this unexpected phenomenon, we return to the expression of mild solutions given in (\ref{System-Integral}). The key distinction between the fractional and classical cases lies in the kernels $h_{\alpha_i}(t,\cdot)$ and $h_2(t,\cdot)$. In estimate (\ref{Key-Estimate-1}) of Proposition \ref{Key-Proposition} below, we rigorously prove that $h_{\alpha_i}(t,\cdot)$ converges to $h_2(t,\cdot)$ as $\alpha_i\to 2$, achieving an optimal convergence rate proportional to $|2-\alpha_i|$. As a result, the overall convergence rate of solutions arises from the interplay between the convergence rate of the initial data and the intrinsic convergence rate of the kernels in the mild formulation.

\medskip

Incidentally, we note that estimate (\ref{Key-Estimate-1}) includes the weight $t^\eta$ in the time variable. This (technical) term is essential for controlling the difference $h_{\alpha_i}(t,\cdot)-h_2(t,\cdot)$ in a suitable norm when $\alpha_i$ is close to the limiting value of $2$, as described in assumption (\ref{Condition-alpha}). Consequently, our main estimate (\ref{Conv-Rate-Sol}) is also expressed with this weight.
For more details, please refer to Remark \ref{Rmk:weight-time}. 

\medskip

Although the weight $t^\eta$ in the estimate (\ref{Conv-Rate-Sol}) is essentially technical, it can be easily removed for strictly positive times. Specifically, for any fixed $0<\varepsilon\ll T_2$, from estimate (\ref{Conv-Rate-Sol}), we obtain the following pointwise-in-time estimate:
\begin{equation}\label{Conv-Rate-Sol-Epsilon}
\| u_{\alpha_i}(t,\cdot)-u_{2,i}(t,\cdot)\|_{H^s} \leq \frac{{\bf C}}{\varepsilon^\eta}\Big( 1+T^{\eta+1}_{2}\Big)\, \max_{i=1,\cdots,n} F_i (|2-\alpha_i|),
\end{equation}
which holds for $i=1,\cdots,n$ and for any time $\varepsilon \leq t \leq T_2$. In this estimate, we clearly observe the convergence rate of solutions to the limiting solution in the classical diffusion equation as the initial value problem evolves and we move slightly away from the initial data. For simplicity, from now on, we will focus on this estimate instead of estimate (\ref{Conv-Rate-Sol}).

\medskip  

On the other hand, we note that the existence of global-in-time solutions to the integral coupled system (\ref{System-Integral}) and similar systems with quadratic nonlinearities remains, in general, a highly nontrivial open question. Please refer to \cite{PG1} for a detailed discussion. This question lies completely outside the scope of this article. However, we mention that, if global-in-time solutions exist for any time $T>0$, the estimate (\ref{Conv-Rate-Sol-Epsilon}) can also be extended to the time interval $[\varepsilon,T]$. See Section \ref{Sec:Step-2} for further details. 

\medskip 

The convergence proven in estimate (\ref{Conv-Rate-Sol-Epsilon}) also allows us to establish an analogous result in other functional spaces. For instance, for certain ranges of the parameters $1\leq p \leq +\infty$ and $\sigma>0$, we can consider the Lebesgue spaces $L^p(\R^d)$ and the Sobolev spaces $\dot{W}^{\sigma,p}(\R^d)$. 
\begin{Corollaire}\label{Corollary}With the same assumptions as in Theorem \ref{Th-Main}, and for fixed  $0<\varepsilon \ll T_2$, the following estimates hold:
\begin{equation}\label{Conv-Rate-Sol-Lp}
\sup_{\varepsilon \leq t \leq T_2} \| u_{\alpha_i}(t,\cdot)-u_{2,i}(t,\cdot)\|_{L^p} \leq \frac{{\bf C}}{\varepsilon^\eta}\Big( 1+T^{\eta+1}_{2}\Big)\, \max_{i=1,\cdots,n} F_i (|2-\alpha_i|), \qquad 2\leq p \leq +\infty,
\end{equation}	
and
\begin{equation}\label{Conv-Rate-Sol-Wsp}
\sup_{\varepsilon \leq t \leq T_2} \| u_{\alpha_i}(t,\cdot)-u_{2,i}(t,\cdot)\|_{\dot{W}^{\sigma,p}} \leq \frac{{\bf C}}{\varepsilon^\eta}\Big( 1+T^{\eta+1}_{2}\Big)\, \max_{i=1,\cdots,n} F_i (|2-\alpha_i|), \qquad 0<\sigma<s, \ \ 2\leq p<+\infty.
\end{equation}
\end{Corollaire}

We thus obtain a wide range of parameters $p$ and $\sigma$ for which these convergences hold. Certain values of these parameters are of particular interest, as the general coupled system (\ref{System}) includes, as a special case, some real-world relevant models. 

\subsection{Applications to related models}\label{Sec:Applications} Here, we briefly discuss the application of the estimates (\ref{Conv-Rate-Sol-Epsilon}), (\ref{Conv-Rate-Sol-Lp}) and (\ref{Conv-Rate-Sol-Wsp}) in the particular models introduced in Section  \ref{Sec:Models}. 

\medskip

For the Navier-Stokes equations (\ref{NS}), the estimate (\ref{Conv-Rate-Sol-Epsilon}) shows the convergence from the fractional diffusion case to the classical diffusion case, in the strong topology of the space $H^s(\Rt)$ with $s>3/2$. In particular, this result improves our previous work \cite{Jarrin_L2024}, where a similar convergence was obtained in the weaker topology of the space $L^\infty(\Rt)$. 

\medskip

Also, for the Navier-Stokes equations (\ref{NS}), in estimate (\ref{Conv-Rate-Sol-Lp}), setting $p=2$ yields a convergence result in the space $L^\infty([\varepsilon, T_2], L^2(\Rt))$. Similarly, in estimate (\ref{Conv-Rate-Sol-Wsp}), setting $\sigma=1$ and $p=2$ provides a convergence result in the space $L^\infty([\varepsilon,T_2], \dot{H}^1(\Rt))$, which is continuously embedded in $L^2([\varepsilon,T_2], \dot{H}^1(\Rt))$. Consequently, we obtain a convergence result in the strong topology of the space $L^\infty_t L^2_x \cap L^2_t \dot{H}^1_x$, which is of particular interest as it represents the natural energy space for these equations. It is worth comparing this result with \cite{Dlotko}. In that article, for a fractional power $\alpha > 5/2$, the author considers the following regularized equation:
\[  \partial_t \vu - (\Delta +\delta(-\Delta)^{\frac{\alpha}{2}} ) \vu + \P\, \text{div}(\vu \otimes \vu)=0, \quad \text{div}(\vu)=0,\]
 and shows that, as $\delta \to 0$, solutions to this equation converge to a weak Leray solution of the classical Navier-Stokes equations in the weak topology of the energy space.
 
 \medskip

On the other hand, the new results stated in estimates (\ref{Conv-Rate-Sol-Epsilon}), (\ref{Conv-Rate-Sol-Lp}), and (\ref{Conv-Rate-Sol-Wsp}) hold for the Magneto-hydrodynamic equations (\ref{MHD}), the Boussinesq system (\ref{Boussinesq}), and the Keller-Segel system (\ref{Keller-Segel}), which, to the best of our knowledge, have not been studied before. In particular, weak and mild solutions to the classical Boussinesq system (\ref{Boussinesq}) have been studied in the setting of $L^p$-spaces; see, for instance, \cite{Brandolese,Cannon}. In estimate (\ref{Conv-Rate-Sol-Lp}), we obtain a fractional approximation to these mild solutions in the strong topology of the $L^p$-spaces, which can be seen as a complement to these previous works.

\medskip

\subsection{Other possible applications} As mentioned below Theorem \ref{Th-Main}, the key idea for proving this result is to first study the convergence of the kernel $h_{\alpha_i}(t,\cdot) \to h_2(t,\cdot)$ as $\alpha_i \to 2$. Specifically, this convergence is established in Proposition \ref{Key-Proposition}, where we prove the main technical estimate:
\[ t^\eta \| \widehat{h}_{\alpha_i}(t,\cdot)- \widehat{h}_2(t,\cdot)\|_{L^\infty} \lesssim |\alpha_i - 2|, \]
which, using the characterization of the $H^s-$norm by the Fourier transform, allows us to write
\[ \| (h_{\alpha_i}(t,\cdot) - h_2(t,\cdot))\ast f \|_{H^s} \lesssim  t^{-\eta}\, |\alpha_i-2|\, \|f\|_{H^s}.  \]

This approach also enables us to prove analogous results to Theorem \ref{Th-Main} in the context of other functional spaces characterized by the Fourier transform, such as the Besov spaces $B^{s,p}_{q}(\R^d)$, the pseudo-measures space $PM^s(\R^d)$, and the Gevrey class $G^{s}_{\beta}(\R^d)$. These spaces are commonly used in the qualitative analysis of the Navier-Stokes equations (\ref{NS}) and other related models discussed in Section \ref{Sec:Models}. For more details, please refer to \cite{PGLibro} and the references therein. 

\medskip

On the other hand, this approach could also be applied to other equations that are not included in the general system (\ref{System}). Among these is the nonlinear reaction-diffusion model studied in \cite{Xu-Ca}:
\[ \partial_t u + (-\Delta)^{\frac{\alpha}{2}} = f(u)+h, \]
and its linear version (when $f(\cdot)\equiv 0$)
\[ \partial_t u + (-\Delta)^{\frac{\alpha}{2}} u = h, \]
considered in  \cite{Biccari}.  For these equations and other related models, the convergence of the fractional diffusion case to the classical diffusion case was established in the weak topology of Sobolev spaces, essentially using compactness arguments. In this context, under suitable hypotheses on $f(\cdot)$ and $h(\cdot)$, this convergence appears to be improvable to the strong topology, with explicit convergence rates. Moreover, given that these equations admit global-in-time solutions and an inherent notion of global attractors, it would be interesting in future works to study this strong convergence for the global attractors. 

\medskip

{\bf Organization of the article and notation}. In Section \ref{Sec:LWP}, we prove Proposition \ref{Prop:LWP}, and the entirety of Section \ref{Sec:Proof-Main-Th} is devoted to a step-by-step proof of Theorem \ref{Th-Main}. On the other hand, the notations $\mathcal{F}(\varphi)$ and $\widehat{\varphi}$ both represent the Fourier transform of $\varphi$. Finally, we will use $C$ to denote a generic positive constant, which may vary from one line to another.

%%%%%%%%%%%%%%%%%%%%%
\section{Proof of Proposition \ref{Prop:LWP}}\label{Sec:LWP}
Given $\alpha_i>1$, for all $i = 1, \cdots, n$, we leverage the structure of the mild solution of system \eqref{System} as defined in \eqref{System-Integral} to consider the linear form
\begin{align*}
	%\textbf{A}(\vu_\alpha) = 
	\textbf{A} \big(u_{\alpha_1}, \cdots, u_{\alpha_n}\big)= \big(A_1 (u_{\alpha_1}, \cdots, u_{\alpha_n}), \cdots, A_n (u_{\alpha_1}, \cdots, u_{\alpha_n})\big),
\end{align*}
with $A_i$ defined, for every $i = 1, \cdots, n$, by the following expression:
\begin{align*}
	A_i(u_{\alpha_1}, \cdots, u_{\alpha_n}) := \sum_{j=1}^{n} \int_{0}^{t} h_{\alpha_i}(t-\tau,\cdot)\ast L_{i,j}(u_{\alpha_j})(\tau,\cdot)d\tau.
\end{align*}
Moreover, we consider the bilinear form
\begin{align*}
	%\textbf{B} \big(\vu_{\alpha}, \vv_{\alpha}\big) &= 
	&\textbf{B} \big( (u_{\alpha_1}, \cdots, u_{\alpha_n}), (v_{\alpha_1}, \cdots, v_{\alpha_n})\big)\\ = &\Big(B_1\big( (u_{\alpha_1}, \cdots, u_{\alpha_n}), (v_{\alpha_1}, \cdots, v_{\alpha_n}) \big), \cdots, B_n\big( (u_{\alpha_1}, \cdots, u_{\alpha_n}), (v_{\alpha_1}, \cdots, v_{\alpha_n}) \big)\Big),
\end{align*}
where each $B_i$  is defined as follows:
\begin{align*}
	B_i\big((u_{\alpha_1}, \cdots, u_{\alpha_n}), (v_{\alpha_1}, \cdots, v_{\alpha_n})\big) := \sum_{j=1}^{n}\sum_{k=1}^{n}\int_{0}^{t}h_{\alpha_i}(t-\tau,\cdot)\ast Q_{i,j,k}(u_{\alpha_j} v_{\alpha_k})(\tau,\cdot) d\tau.
\end{align*}

Thus, by setting $e = (u_{\alpha_1}, \cdots, u_{\alpha_n})$ and $e_0 = \big(h_{\alpha_1}(t,\cdot)\ast u_{0,\alpha_1}, \cdots, h_{\alpha_n}(t,\cdot)\ast u_{0,\alpha_n}\big)$, the coupled system of integral equations \eqref{System-Integral} can be written in the form
\begin{align}\label{Picard}
	e = e_0 + \textbf{A}(e) + \textbf{B}(e, e).
\end{align}
Therefore, to construct a solution of \eqref{Picard}, we employ the following version of the Picard's iteration scheme. 	For a proof,  we refer to \cite{Chamorro-Yangari} (proof of Theorem 3.2 in Appendix).
\begin{Lemme}\label{Lemma_Picard} Let $(E, \norm{\ \cdot\ }_E)$ be a Banach space and let $e_0 \in E$ be an initial data such that $\norm{e_0}_E \leq \delta$. Moreover, let $\emph{\textbf{A}}: E \to E$ be a linear form and let $\emph{\textbf{B}}: E \times E \to E$ be a bilinear form such that, for all $e, f \in E$,
	\begin{equation}\label{Cont-Picard}
		\norm{\emph{\textbf{A}}(e)}_E \leq C_\emph{\textbf{A}} \norm{e}_E \qquad \text{and} \qquad \norm{\emph{\textbf{B}}(e,f)}_E \leq C_\emph{\textbf{B}} \norm{e}_E \norm{f}_E.
	\end{equation}		
	
	If the constants $C_\emph{\textbf{A}} > 0$ and $C_\emph{\textbf{B}} > 0$ satisfy the relationships:
	\begin{equation}\label{Conditions-Picard}
		0 < 3 C_\emph{\textbf{A}} < 1,\qquad 0 < 9 C_\emph{\textbf{B}}\ \delta < 1, \qquad \text{and} \qquad C_\emph{\textbf{A}} + 6 C_\emph{\textbf{B}}\ \delta < 1,
	\end{equation} 
	then equation \eqref{Picard} admits a unique solution $e$ such that $\norm{e}_E \leq 3\delta$.
\end{Lemme}

Within this framework, for a time $0< T <+\infty$ that we will set later, we consider the Banach space $E= \ds{\prod_{i=1}^{n}\mathcal{C}\big([0,T], H^{s}(\R^d)\big)}$,  endowed with its natural norm $\|(u_{\alpha_1}, \cdots, u_{\alpha_n})\|_E=\ds \sum_{i=1}^{n}\|  u_{\alpha_i}\|_{L^{\infty}_{t} H^{s}_{x}}$. 

\medskip

Let us note that, for all $i = 1, \cdots, n$, the terms involving the initial data $e_0$ are directly estimated as $\left\| h_{\alpha_i}(t,\cdot)\ast u_{0,\alpha_i} \right\|_{L^{\infty}_{t} H^{s}_{x}} \leq C\, \| u_{0,\alpha_i} \|_{H^{s}}$, which yields $\norm{e_0}_E  \leq  C\, \ds \sum_{i=1}^{n} \| u_{0,\alpha_i} \|_{H^{s}}$. 

\medskip

We can now continue verifying the controls in \eqref{Cont-Picard}. Thus,  for $0< t < T$ fixed, the linear terms $A_i$ can be estimated as follows
\begin{align*}
	\left\| A_i\big(u_{\alpha_1}, \cdots, u_{\alpha_n}\big) \right\|_{H^{s}} &= \left\| \sum_{j=1}^{n} \int_{0}^{t} h_{\alpha_i}(t-\tau,\cdot)\ast L_{i,j}(u_{\alpha_j})(\tau,\cdot)d\tau \right\|_{H^{s}}\\ &\leq \sum_{j=1}^{n} \int_{0}^{t} \| h_{\alpha_i}(t-\tau,\cdot) \|_{L^1} \, \| L_{i,j}(u_{\alpha_j})(\tau,\cdot) \|_{H^{s}} d \tau, 
\end{align*}
where, according to \cite[Lemma $2.2$]{Yu-Zhai}, the convolution kernel verifies $\ds{ \| h_{\alpha_i}(t-\tau,\cdot) \|_{L^1} \leq C}$. In addition, for $i,j = 1, \cdots, n$, as the operator $L_{i,j}(\cdot)$ is defined by the symbol $\widehat{\ell_{i,j}}(\xi)$, which is bounded for any $\xi \neq 0$, there exists a constant $c_{i,j} > 0$ such that $\|L_{i,j}(u)\|_{H^s} \leq c_{i,j}\|u\|_{H^s}$.  Thus, we have
\begin{equation*}
\begin{split}
	\sum_{j=1}^{n} \int_{0}^{t} \| h_{\alpha_i}(t-\tau,\cdot) \|_{L^1} \, \| L_{i,j}(u_{\alpha_j})(\tau,\cdot) \|_{H^{s}} d \tau \leq &\,  C \left( \ds{\max_{i,j=1, \ldots, n} c_{i,j}} \right) \left( \int_{0}^{t} d\tau \right) \sum_{j=1}^{n} \| u_{\alpha_j} \|_{L^{\infty}_{t} H^{s}_{x}} \\
	 \leq  &\, C\, T\,  \| (u_{\alpha_1}, \cdots, u_{\alpha_n}) \|_{E}.
	\end{split}
\end{equation*}

On the other hand, for $i,j,k = 1, \cdots, n$, as the operator $Q_{i,j,k}(\cdot)$ is defined by the symbol $\widehat{q}_{i,j,k}(\xi)$, which is a homogeneous function of order 1, for any $\xi \neq 0$  there exists a constant $c_{i,j,k}>0$  such that $|\widehat{q}_{i,j,k}(\xi)|\leq c_{i,j,k}|\xi|$. Denoting by $c= \ds \max_{i,j,k=1,\cdots,n}c_{i,j,k}$,   each bilinear term $B_i$ is estimated as follows
\begin{align*}
	&\left\| B_i\big((u_{\alpha_1}, \cdots, u_{\alpha_n}), (v_{\alpha_1}, \cdots, v_{\alpha_n})\big) \right\|_{H^{s}}\\ \leq  & \sum_{j=1}^{n}\sum_{k=1}^{n} \int_{0}^{t} \left\|  h_{\alpha_i}(t-\tau,\cdot)\ast Q_{i,j,k}(u_{\alpha_j} v_{\alpha_k})(\tau,\cdot) d\tau \right\|_{H^{s}}\\
	\leq & c\,  \sum_{j=1}^{n}\sum_{k=1}^{n} \int_{0}^{t} \| \vec{\nabla} h_{\alpha_i}(t-\tau,\cdot) \|_{L^1} \, \left\| u_{\alpha_j} u_{\alpha_k}  (\tau, \cdot) \right\|_{H^{s}} d \tau. 
\end{align*}
 Recall that from \cite[Lemma $2.2$]{Yu-Zhai}, we have $\ds{ \| \vec{\nabla} h_{\alpha_i}(t-\tau,\cdot) \|_{L^1} \leq C (t-\tau)^{-\frac{1}{\alpha_i}}}$. Moreover, as $s>d/2$, the product laws in Sobolev spaces $H^s(\R^d)$  yield 
\[  \left\| u_{\alpha_j} u_{\alpha_k}  (\tau, \cdot) \right\|_{H^{s}} \leq c_s \| u_{\alpha_j}(\tau,\cdot)\|_{H^s}\, \| u_{\alpha_k}(\tau,\cdot)\|_{H^s}, \]
with a constant $c_s>0$.   Thus, for $\alpha_i>1$ we obtain
\begin{align*}
	%&\left\| B_i\big(\ (u_{\alpha_i}, \cdots, u_{\alpha_n}), (v_{\alpha_i}, \cdots, v_{\alpha_n})\ \big) \right\|_{H^{s}}\\
	& c\,  \sum_{j=1}^{n}\sum_{k=1}^{n} \int_{0}^{t} \| \vec{\nabla} h_{\alpha_i}(t-\tau,\cdot) \|_{L^1} \, \left\| u_{\alpha_j} u_{\alpha_k}  (\tau, \cdot) \right\|_{H^{s}} d \tau\\
	\leq &\,  c\,c_s \ \left( \int_{0}^{t} (t-\tau)^{-\frac{1}{\alpha_i}} \, d\tau\right)  \sum_{j=1}^{n} \| u_{\alpha_j} \|_{L^{\infty}_{t} H^{s}_{x}}\,  \sum_{k=1}^{n} \| v_{\alpha_k} \|_{L^{\infty}_{t} H^{s}_{x}}  \\
	\leq &\,  C\ \frac{T^{1-\frac{1}{\alpha_i}}}{1-\frac{1}{\alpha_i}}\,  \|(u_{\alpha_1}, \cdots, u_{\alpha_n})  \|_{E}\, \| (v_{\alpha_1}, \cdots, v_{\alpha_n}) \|_{E}.
\end{align*}

In consequence,  the following estimates hold:
\begin{align*}
	\norm{\textbf{A} \big(u_{\alpha_1}, \cdots, u_{\alpha_n}\big)}_E \leq C nT\, \| (u_{\alpha_i}, \cdots, u_{\alpha_n})\|_{E},
\end{align*}
and
\begin{align*}
	\norm{\textbf{B} \big((u_{\alpha_1}, \cdots, u_{\alpha_n}), (v_{\alpha_1}, \cdots, v_{\alpha_n})\big)}_E \leq C \left(\sum_{i = 1}^{n} \frac{T^{1-\frac{1}{\alpha_i}}}{1-\frac{1}{\alpha_i}}\right)   \|(u_{\alpha_1}, \cdots, u_{\alpha_n})  \|_{E}\, \| (v_{\alpha_1}, \cdots, v_{\alpha_n}) \|_{E}.
\end{align*}

\medskip

Therefore, within the setting of Lemma \ref{Lemma_Picard}, we consider $$\delta := \ds C\,\sum_{j=1}^{n} \| u_{0,\alpha_j} \|_{H^{s}}, \qquad \ds C_\textbf{A} := Cn 
T, \qquad \text{and} \qquad C_\textbf{B} := C \left(\sum_{i = 1}^{n} \frac{T^{1-\frac{1}{\alpha_i}}}{1-\frac{1}{\alpha_i}}\right).$$ Hence, we must verify the whole set of conditions in (\ref{Conditions-Picard}). Let us note that the first condition is verified as long as 
\begin{equation*}
	0< 3 Cn T < 1.
\end{equation*} 
Once we have this constraint, observe that the second and the third condition in (\ref{Conditions-Picard}) simultaneously hold as long as 
\[ 0< 9C \left(\sum_{i = 1}^{n} \frac{T^{1-\frac{1}{\alpha_i}}}{1-\frac{1}{\alpha_i}}\right) \left( \ds \sum_{j=1}^{n} \| u_{0,\alpha_j} \|_{H^{s}}\right) <1,\]
which is ultimately verified by assuming the constrain
\begin{equation*}
	9 C \left( \frac{T^{1-\frac{1}{\alpha_i}}}{1-\frac{1}{\alpha_i}}\right)  \| u_{0,\alpha_j} \|_{H^{s}}< \frac{1}{n^2}, \qquad \text{for all } i,j= 1, \cdots, n.
\end{equation*}
Finally, we set the existence time $T=T_{\alpha}$ as in expression (\ref{T-alpha}), and Proposition \ref{Prop:LWP} is proven. 
%%%%%%%%%%%%%%%%%%%%%%
\section{Proof of Theorem \ref{Th-Main}}\label{Sec:Proof-Main-Th}
\subsection{Convergence of kernels} As mentioned in the introduction, our main estimate (\ref{Conv-Rate-Sol}) strongly depends on the convergence of the kernels $ h_{\alpha_i}(t,\cdot)\to h_2(t,\cdot)$ as $\alpha_i \to 2$. Thus, one of the \emph{key tools} of this paper is the following  result:
\begin{Proposition}\label{Key-Proposition}  Fix $0<\delta < \frac{1}{6}$, and let  $\frac{1}{4}<\eta<\frac{1}{2}$  be the  quantity defined in expression (\ref{Quantities}). There exists a constant $C=C(\delta)>0$, such that, for   every time $0<T<+\infty$,   the following estimate holds: 
		\begin{equation}\label{Key-Estimate-1}
	\sup_{0\leq  t \leq T} t^\eta  \left\Vert \widehat{h}_{\alpha_i}(t,\cdot)-\widehat{h}_2(t,\cdot) \right\Vert_{L^\infty} \leq  C \left(1+T^{\eta+1}\right)|2- \alpha_i|,
	\end{equation} 
	for any $\alpha_i$ satisfying $2-\delta<\alpha_i < 2+\delta$, with $i=1,\cdots, n$. Additionally, define the quantity $\ds{\kappa:= \frac{3+4\delta}{4-2\delta}}$, which verifies $\frac{3}{4}<\kappa<1$. Then, it holds:
		\begin{equation}\label{Key-Estimate-2}
	\sup_{0\leq  t \leq T} t^\kappa  \left\Vert |\xi| \left( \widehat{h}_{\alpha_i}(t,\cdot)-\widehat{h}_2(t,\cdot) \right) \right\Vert_{L^\infty} \leq  C \left(1+T^{\kappa+1}\right)|2- \alpha_i|, 
	\end{equation}
	for any $2-\delta<\alpha_i < 2+\delta$. 
\end{Proposition}
\begin{proof} To simplify our notation, we will omit the subscript in $\alpha_i$ and simply write $\alpha$.   We start by proving the estimate (\ref{Key-Estimate-1}). For fixed  $0 \leq t \leq T$ and $\xi \neq 0$,  we define the following function depending on the variable $\alpha$:
\[f_{t,\xi}(\alpha)=e^{-t |\xi|^{\alpha}}, \quad 2-\delta < \alpha < 2+\delta. \]
Computing its derivative we obtain
\[ \frac{d}{d\alpha} f_{t,\xi}(\alpha)= -t e^{-t |\xi|^{\alpha}}|\xi|^{\alpha} \ln |\xi|.\]
Then,  applying the mean value theorem, we write
\begin{equation*}
| f_{t,\xi}(\alpha) - f_{t,\xi}(2)| \leq \left\| \frac{d}{d\alpha} f_{t,\xi}\right\|_{L^\infty_{\alpha}(2-\delta,2+\delta)}\,|2-\alpha|.
\end{equation*}
We thus get
\begin{equation}\label{Estim-01}
\left\| \widehat{h}_{\alpha}(t,\cdot)-\widehat{h}_2(t,\cdot)  \right\|_{L^\infty} \leq \left\| \left\| \frac{d}{d\alpha} f_{t,\xi}\right\|_{L^\infty_\alpha(2-\delta,2+\delta)} \right\|_{L^\infty_\xi (\Rt)}\, |2-\alpha|=:(A)\, |2-\alpha|,
\end{equation}
where we must estimate the term $(A)$. To do this, we write
\[ (A)\leq   \left\| \left\| \frac{d}{d\alpha} f_{t,\xi}\right\|_{L^\infty_\alpha(2-\delta,2+\delta)} \right\|_{L^\infty_\xi (|\xi|\leq 1)} + \left\| \left\| \frac{d}{d\alpha} f_{t,\xi}\right\|_{L^\infty_\alpha(2-\delta,2+\delta)}\, \right\|_{L^\infty_\xi (|\xi|>1)}=:(A_1)+(A_2).  \]
To estimate the term $(A_1)$, since  $|\xi|\leq 1$ and   $1<2-\delta <\alpha$, it follows that $|\xi|^\alpha \leq |\xi|$  and $\ds{\lim_{|\xi|\to 0^{+}}|\xi| \ln|\xi|=0}$. We then  obtain
\begin{equation*}
(A_1) \leq \,  \sup_{|\xi|\leq 1} \left( \sup_{1<\alpha < 2+\delta} \left| t e^{-t |\xi|^\alpha}|\xi|^\alpha \ln |\xi| \right|  \right) \leq t\sup_{|\xi|\leq 1} \left( |\xi| \ln|\xi|  \right)\leq t\,C. 
\end{equation*}
To estimate the term $(A_2)$, since $2-\delta <\alpha<2+\delta$ and $|\xi|>1$, we obtain $|\xi|^{2-\delta}\leq |\xi|^\alpha \leq |\xi|^{2+\delta}$. In addition, always since $|\xi|>1$, it follows that   $\ln|\xi| \leq |\xi|^{\frac{1}{2}}$.  We  then  write
\begin{equation*}
\begin{split}
(A_2) = &\,   \sup_{|\xi|>1}\left( \sup_{2-\delta<\alpha<2+\delta}\left| t e^{-t |\xi|^\alpha}|\xi|^\alpha \ln |\xi| \right| \right)  \leq  t \,  \sup_{|\xi|>1}\left( e^{-t |\xi|^{2-\delta}}|\xi|^{2+\delta}|\xi|^{\frac{1}2} \right)\\
\leq &\, t  \, \sup_{|\xi|>1} \left( e^{-\left| t^{\frac{1}{2-\delta}}\xi \right|^{2-\delta}} |\xi|^{\frac{5+2\delta}{2}} \right)  \leq  t^{1-\frac{5+2\delta}{4-2\delta}} \sup_{\xi \in \R^d} \left( e^{-\left| t^{\frac{1}{2-\delta}}\xi \right|^{2-\delta}} \left|t^{\frac{1}{2-\delta}}\xi\right|^{\frac{5+2\delta}{2}} \right).
\end{split}
\end{equation*}
\begin{Remarque}\label{Rmk:weight-time}
Note that for any $\delta >0$, it holds that $1-\frac{5+2\delta}{4-2\delta}<0$. Consequently, when $t\to 0^{+}$, we lose control over the term $t^{1-\frac{5+2\delta}{4-2\delta}}$.  In this context, the quantity $\eta$, defined in (\ref{Quantities}), is expressly  introduced to obtain $\eta + 1 - \frac{5+2\delta}{4-2\delta}=0$.   
\end{Remarque}
Then, we get 
\[ t^\eta (A_2)\leq C.  \] 

Returning to  (\ref{Estim-01})  and combining the  estimates for the terms $(A_1)$ and $(A_2)$, we can write 
\begin{equation*}
t^\eta  \left\| \widehat{h}_{\alpha}(t,\cdot)-\widehat{h}_2(t,\cdot)  \right\|_{L^\infty} \leq t^\eta (A) |2-\alpha| \leq  \big( t^\eta (A_1)+t^\eta (A_2)\big) |2-\alpha| \leq C(t^{\eta+1}+1)|2-\alpha|,
\end{equation*}
from which the desired estimate (\ref{Key-Estimate-1}) follows. 

\medskip

The second estimate (\ref{Key-Estimate-2})  follows from very similar  computations, this time by considering the  function 
\[ f_{t,\xi}(\alpha)= |\xi|e^{-t|\xi|^\alpha}, \quad 2-\delta < \alpha < 2+\delta. \]
In particular, the term $(A_2)$ can be expressed as 
\[ (A_2) \leq t^{\kappa+1} \, \sup_{|\xi|>1} \left( e^{-\left| t^{\frac{1}{2-\delta}}\xi \right|^{2-\delta}} |\xi|^{\frac{7+2\delta}{2}} \right)  \leq  t^{\kappa+1-\frac{7+2\delta}{4-2\delta}} \sup_{\xi \in \R^d} \left( e^{-\left| t^{\frac{1}{2-\delta}}\xi \right|^{2-\delta}} |t^{\frac{1}{2-\delta}}\xi|^{\frac{7+2\delta}{2}} \right), \]
where  the quantity $\kappa$  defined above    implies that  $\ds{\kappa+1-\frac{7+2\delta}{4-2\delta}=0}$.  Thus,  Proposition \ref{Key-Proposition} is proven.  
\end{proof} 

\subsection{Proof of Theorem \ref{Th-Main}} 
Recall that for the solution $(u_{\alpha_1}, \ldots, u_{\alpha_n})$ to the system (\ref{System-Integral}), obtained in Proposition \ref{Prop:LWP}, each function $u_{\alpha_i}$ exists over the time interval $[0, T_{\alpha }]$. For clarity, we will divide the proof of the estimate (\ref{Conv-Rate-Sol}) into two main steps. In the first step, we will prove this estimate over an interval $[0, T]$, where the time $T$ is suitably small. Then, in the second step, we will show that this estimate extends to the entire time interval $[0, T_2]$, where $T_2$ is defined in (\ref{T_2}) and denotes the existence time of the limit solution $(u_{2,1}, \ldots, u_{2,n})$ in the classical diffusion case.

\subsubsection{Step 1}   From now on, we assume that the fractional powers $\alpha_i$ are sufficiently close to the limiting value of $2$ in the sense described by condition (\ref{Condition-alpha}). This assumption allows to prove the following technical result:
\begin{Lemme} Let $T_{\alpha}>0$ denote the time obtained in expression (\ref{T-alpha}). Assume (\ref{Convergence-Data}), and  for a parameter $0 < \delta < \frac{1}{6}$, assume that condition (\ref{Condition-alpha}) holds. Then, there exists a constant $T_0 > 0$, depending  on $\delta$, such that
	\begin{equation}\label{Control-Time}
	T_0 \leq T_{\alpha}, \quad \text{	for all $2 - \delta < \alpha_i < 2 + \delta$}. 
	\end{equation}
\end{Lemme}
\begin{proof}
By estimate (\ref{Convergence-Data}) and the fact that $|2-\alpha_j|<\delta$, we write $\ds{\big| \| u_{0,\alpha_j} \|_{H^s} -  \| u_{0,2,j}\|_{H^s} \big| < \delta^{\beta_j} \leq \delta^\beta}$, with $\ds{\beta:=\min_{i=1,\cdots,n}\beta_i}$. Then, we obtain 
\[  \| u_{0,\alpha_j}\|_{H^s}< \| u_{2,0,j}\|_{H^s}+\delta^\beta. \]
 Returning to expression (\ref{T-alpha}), from this last inequality, we can write
\[ \left( \frac{1-\frac{1}{\alpha_i} }{9n^2  C   \| u_{0,2,j} \|_{H^{s}}+ \delta^\beta }\right) ^{\frac{\alpha_i}{\alpha_i-1}} \leq  \left( \frac{1-\frac{1}{\alpha_i} }{9n^2 C  \| u_{0,\alpha_j} \|_{H^{s}}  }\right) ^{\frac{\alpha_i}{\alpha_i-1}}. \]
In addition, as $2-\delta < \alpha_i$, we have $1- \frac{1}{2-\delta}<1-\frac{1}{\alpha_i}$, and we obtain
\[ \left( \frac{1-\frac{1}{2-\delta} }{9n^2  C   \| u_{0,2,j} \|_{H^{s}}+ \delta^\beta }\right) ^{\frac{\alpha_i}{\alpha_i-1}} \leq \left( \frac{1-\frac{1}{\alpha_i} }{9n^2  C \ \| u_{0,\alpha_j} \|_{H^{s}}  }\right) ^{\frac{\alpha_i}{\alpha_i-1}}. \]

For simplicity, we introduce  the quantity  
\[ \Phi_j := \frac{1-\frac{1}{2-\delta} }{9n^2  C  \| u_{0,2,j} \|_{H^{s}}+ \delta^\beta}, \]
and we write
\[ \Phi_j^{\frac{\alpha_i}{\alpha_i-1}} \leq \left( \frac{1-\frac{1}{\alpha_i} }{9n^2C \| u_{0,\alpha_j} \|_{H^{s}} }\right) ^{\frac{\alpha_i}{\alpha_i-1}}, \]
where we still need to study the term $\ds{\Phi_j^{\frac{\alpha_i}{\alpha_i-1}}}$. As $2-\delta < \alpha_i < 2+\delta$, it follows that $\ds{\frac{2-\delta}{1+\delta}< \frac{\alpha_i}{\alpha_i-1}<\frac{2+\delta}{1-\delta}}$.  Then, if the quantity $\Phi_j$ introduced above satisfies   $\Phi_j<1$, then we obtain $\ds{\Phi_j^{\frac{2+\delta}{1-\delta}}\leq \Phi_j^{\frac{\alpha_i}{\alpha_i-1}}}$, meanwhile if $\Phi_j\geq 1$ we obtain 
$\ds{\Phi_j^{\frac{2-\delta}{1+\delta}} \leq \Phi_j^{\frac{\alpha_i}{\alpha_i-1}}}$.  Consequently, we can write
\[ \min\left( \Phi_j^{\frac{2+\delta}{1-\delta}}, \, \Phi_j^{\frac{2-\delta}{1+\delta}} \right) \leq \Phi_j^{\frac{\alpha_i}{\alpha_i-1}}. \]
 In this way, using again the expression (\ref{T-alpha}), we finally obtain the lower bound
\[ T_0:= \frac{1}{2} \left[ \frac{1}{3nC}, \  \min_{j=1,\cdots,n} \left( \min\left( \Phi_j^{\frac{2+\delta}{1-\delta}}, \, \Phi_j^{\frac{2-\delta}{1+\delta}} \right) \right) \right] \leq T_{\alpha}. \]
\end{proof}	
We thus obtain that each component of the solution $(u_{\alpha_1}, \ldots, u_{\alpha_n})$ is at least defined over the time interval $[0, T_0]$, where the time $T_0$ is independent of $\alpha_i$.  Next, we fix a time $T < \min(1, T_0)$, which will be chosen sufficiently small later, and prove that the estimate (\ref{Conv-Rate-Sol}) holds over the interval $[0, T]$.  
 
\medskip

 Using the integral equations in the system (\ref{System-Integral}),  and with the quantity $\eta$ defined in (\ref{Quantities}),  we write for $i=1,\cdots,n$, and for $0\leq t \leq T$:
\begin{equation}\label{Ineq1}
\begin{split}
t^{\eta}\| u_{\alpha_i}(t,\cdot)-u_{2,i}(t,\cdot)\|_{H^s}\leq &\,\, \underbrace{t^{\eta}\left\| h_{\alpha_i}(t,\cdot)\ast u_{0,\alpha_i} - h_2(t,\cdot)\ast u_{0,2,i} \right\|_{H^s}}_{I_1} \\
&\, \,+   \sum_{j=1}^{n}\sum_{k=1}^{n}\,t^{\eta}\, \left\|  \int_{0}^{t} h_{\alpha_i} (t-\tau,\cdot) \ast Q_{i,j,k}(u_{\alpha_j} u_{\alpha_k})(\tau,\cdot) d \tau  \right. \\
&\,\, \underbrace{\qquad  \left.  -  \int_{0}^{t} h_2(t-\tau,\cdot)\ast Q_{i,j,k}(u_{2,j} u_{2,k})(\tau,\cdot)d \tau   \right\|_{H^s}}_{I_2} \\
&\,\, + \sum_{j=1}^{n}\, t^{\eta}\, \left\|  \int_{0}^{t} h_{\alpha_i}(t-\tau,\cdot)\ast L_{i,j}(u_{\alpha_j})(\tau,\cdot)d \tau \right. \\  
&\underbrace{\qquad  \left.-   \int_{0}^{t} h_2(t-\tau,\cdot) \ast L_{i,j}(u_{2,j})(\tau,\cdot)d \tau \right\|_{H^s}}_{I_3}.
\end{split}
\end{equation} 
We must estimate the terms $I_1, I_2, I_3$.  For the sake of clarity, we will consider each term separately.

\medskip

\emph{Term $I_1$}.  We write
\begin{equation*}
I_1 \leq \,  t^\eta \left\| \left( h_{\alpha_i}-h_2\right)(t,\cdot)\ast u_{0,\alpha_i}\right\|_{H^s}+ t^\eta \left\| h_2(t,\cdot)\ast \left(u_{0,\alpha_i} - u_{0,2,i} \right)\right\|_{H^s} =:\, I_{1,1}+I_{1,2}.
\end{equation*}
To control the first term $I_{1,1}$,   recall that  by our assumption (\ref{Convergence-Data}), for $i=1,\cdots,n$, the family of initial data $\{ u_{0,\alpha_i} :  \,  2-\delta<\alpha_i<2+\delta, \, \alpha_i\neq 2 \}$ is bounded in $H^s(\R^d)$. Consequently,  there exists a quantity $M_{i,\delta}>0$,   depending on $\| \vu_{0,2,i}\|_{H^s}$ and $\delta$, such that $\ds{\sup_{2-\delta<\alpha_i<2+\delta}\| \vu_{0,\alpha_i}\|_{H^s}\leq M_{i,\delta}}$. In addition, to simplify our notation, we will denote $M_\delta:= \ds{\max_{i=1,\cdots,n}M_{i,\delta}}$. Then, we obtain the uniform control 
\begin{equation}\label{Control-Data}
\max_{i=1,\cdots,n}\left(\sup_{2-\delta<\alpha_i<2+\delta}\| \vu_{0,\alpha_i}\|_{H^s}\right)\leq M_{\delta}.
\end{equation}
Moreover,  applying estimate (\ref{Key-Estimate-1}), for $0\leq t \leq T$, we have
\begin{equation*}
\begin{split}
I_{1,1}  \leq &\, t^{\eta} \left\| \left( \widehat{h}_{\alpha_i}-\widehat{h}_{2} \right)(t,\cdot)\,\widehat{u}_{0,\alpha_i}\, (1+|\xi|^2)^{\frac{s}{2}} \right\|_{L^2}\\
\leq &\, \left( t^{\eta}  \left\| \widehat{h}_{\alpha_i}(t,\cdot)-\widehat{h}_{2}(t,\cdot) \right\|_{L^\infty} \right) \| u_{0,\alpha_i}\|_{H^s}\\
\leq &\, C\left(1+T^{\eta+1}\right)\,|2-\alpha_i|\, M_{\delta}.
\end{split}
\end{equation*}
Second term $I_{1,2}$ is directly controlled by estimate (\ref{Convergence-Data})  as follows:
\begin{equation*}
I_{1,2}\leq \,  t^\eta \left\| h_2(t,\cdot)\ast \left(u_{0,\alpha_i} - u_{0,2,i} \right)\right\|_{H^s} \leq {C \| u_{0,\alpha_i} - u_{0,2,i}\|_{H^s}} \leq C\,  {\bf c}  |2-\alpha_i|^{\beta_i}.
\end{equation*} 
Gathering these estimates, and using the function $F_i(\cdot)$ defined in (\ref{Function}),  we  get
\begin{equation}\label{I1}
\begin{split}
I_1 \leq &\,  {\bf c} M_\delta\, C\left(1+T^{\eta+1}\right)  \, \max\big( |2-\alpha_i|, |2-\alpha_i|^{\beta_i} \big)=:\mathfrak{C}_{1}\left(1+T^{\eta+1}\right)\,
 F_i(|2-\alpha_i|)\\
 \leq &\, \mathfrak{C}_{1}\left(1+T^{\eta+1}\right)\,
 \max_{i=1,\cdots,n}F_i(|2-\alpha_i|).
 \end{split}
\end{equation}

\emph{Term $I_2$}.  We write
\begin{equation}\label{I2-Previous}
\begin{split}
I_2 \leq & \, \sum_{j=1}^{n}\sum_{k=1}^{n}\, t^{\eta} \left\|  \int_{0}^{t}\left( h_{\alpha_i} - h_2\right)(t-\tau,\cdot) \ast Q_{i,j,k}(u_{\alpha_j} u_{\alpha_k})(\tau,\cdot)  d\tau \right\|_{H^s}\\
&\, + \sum_{j=1}^{n}\sum_{k=1}^{n} t^{\eta} \left\| \int_{0}^{t} h_2(t-\tau,\cdot)\ast  Q_{i,j,k}\big(u_{\alpha_j} u_{\alpha_k}- u_{2,j}u_{2,k}\big) (\tau,\cdot)d \tau   \right\|_{H^s}\\
=:&\, I_{2,1}+I_{2,2}. 
\end{split}
\end{equation}

To control the first term $I_{2,1}$, recall that the operator $Q_{i,j,k}(\cdot)$  is defined in expression (\ref{Term-Q}) by the symbol $\widehat{q}_{i,j,k}(\xi)$. Moreover, this symbol is a homogeneous function of order 1. Therefore, for any $\xi \neq 0$, one has $|\widehat{q}_{i,j,k}(\xi)|\leq c_{i,j,k}|\xi|$, where $c_{i,j,k}>0$ is a constant. We will denote $\ds{C_1:=\max_{i,j,k=1,\cdots,n}c_{i,j,k}}$, and  we then write
\begin{equation*}
\begin{split}
I_{2,1} \leq &\, \sum_{j=1}^{n}\sum_{k=1}^{n} t^\eta \, \int_{0}^{t} \left\| \left( \widehat{h}_{\alpha_i} - \widehat{h}_2\right)(t-\tau,\cdot) \widehat{q}_{i,j,k}(\xi) (\widehat{u_{\alpha_j} u_{\alpha_k}})(\tau,\cdot) (1+|\xi|^2)^{\frac{s}{2}} \right\|_{L^2}\, d\tau\\
\leq& \, C_1\,  \sum_{j=1}^{n}\sum_{k=1}^{n} t^\eta \, \int_{0}^{t} \left\| |\xi|\left( \widehat{h}_{\alpha_i} - \widehat{h}_2\right)(t-\tau,\cdot)  (\widehat{u_{\alpha_j} u_{\alpha_k}})(\tau,\cdot) (1+|\xi|^2)^{\frac{s}{2}} \right\|_{L^2}\, d\tau\\
\leq &\, C_1\,  \sum_{j=1}^{n}\sum_{k=1}^{n} t^\eta \, \int_{0}^{t} \left\| |\xi|\left( \widehat{h}_{\alpha_i} - \widehat{h}_2\right)(t-\tau,\cdot) \right\|_{L^\infty}\, \left\| u_{\alpha_j} u_{\alpha_k}(\tau,\cdot)  \right\|_{H^s}\, d\tau.
\end{split}
\end{equation*}
For the first term, by estimate (\ref{Key-Estimate-2}) we obtain
\begin{equation}\label{Control-1}
\left\| |\xi|\left( \widehat{h}_{\alpha_i} - \widehat{h}_2\right)(t-\tau,\cdot) \right\|_{L^\infty} \leq (t-\tau)^{-\kappa}\, C(1+T^{\kappa+1})|2-\alpha_i|.
\end{equation}
For the second term, we can prove the following:
\begin{Lemme}  There exists a constant $C_2>0$, which depends  essentially on the initial data $(u_{0,2,1},\cdots, u_{0,2,n})$ and the parameters $\delta,s$, such that  the following uniform bound holds:
	\begin{equation}\label{Control-2}
\max_{j,k=1,\cdots,n}\left(	\sup_{2-\delta<\alpha_j,\alpha_k<2+\delta} \left(\sup_{0\leq \tau \leq T} \| u_{\alpha_j} u_{\alpha_k}(\tau,\cdot)  \|_{H^s}\right)\right)\leq C_2.
	\end{equation}
	\end{Lemme}
\begin{proof}
Recall 	that the solution $(u_{\alpha_1}, \cdots, u_{\alpha_n}) \subset \mathcal{C}_t H^s_x$ to the integral coupled system (\ref{System-Integral}) is obtained in Proposition \ref{Prop:LWP} using Picard's iterative schema.  Consequently, for each $j=1,\cdots,n$,  one has
\begin{equation*}
\sup_{0\leq \tau \leq T} \| u_{\alpha_j}(\tau,\cdot)\|_{H^s} \leq \sup_{0\leq \tau \leq T_{\alpha_j}} \| u_{\alpha_j}(\tau,\cdot)\|_{H^s} \leq  c_j \, \| u_{0,\alpha_j}\|_{H^s},
\end{equation*}
where $T<T_0\leq T_{\alpha_j}$, and with a constant  $c_j>0$.  Then,  by the control given in (\ref{Control-Data}) and denoting $\ds{c:=\max_{j=1,\cdots,n}c_j}$,  we obtain
\begin{equation}\label{Control-Sol-2}
\sup_{2-\delta<\alpha_j<2+\delta} \left( \sup_{0\leq \tau \leq T} \| u_{\alpha_j}(\tau,\cdot)\|_{H^s}\right) \leq  c_j \, \left(\sup_{2-\delta<\alpha_j<2+\delta}   \| u_{0,\alpha_j}\|_{H^s}  \right) \leq c\,  M_{\delta}.
\end{equation}
In this way, from product laws in Sobolev spaces $H^s(\R^d)$ with $s>d/2$, there exists a constant $c_s>0$ such that  we can write
\begin{equation*}
\begin{split}
&\, \sup_{2-\delta<\alpha_j,\alpha_k<2+\delta} \left(\sup_{0\leq \tau\leq T} \| u_{\alpha_j} u_{\alpha_k}(\tau,\cdot)\|_{H^s} \right) \\
\leq  &\,  \sup_{2-\delta<\alpha_j,\alpha_j<2+\delta} \left(\sup_{0\leq \tau\leq T} c_s \| u_{\alpha_j}(\tau,\cdot)\|_{H^s}\, \|u_{\alpha_k}(\tau,\cdot)\|_{H^s} \right)\\
\leq  &\, c_s \, \left(  \sup_{2-\delta<\alpha_j<2+\delta} \left(\sup_{0\leq \tau\leq T} \| u_{\alpha_j}(\tau,\cdot)\|_{H^s} \right)\right) \times \left( \sup_{2-\delta\alpha_k<2+\delta} \left( \sup_{0\leq \tau\leq T} \|u_{\alpha_k}(\tau,\cdot)\|_{H^s} \right)\right) \\
\leq &\, c_s (c\,M_\delta)^2 =:  C_2,
\end{split}
\end{equation*}
from which we obtain the desired estimate (\ref{Control-2}).
\end{proof}	

\medskip

Having established the controls given in (\ref{Control-1}) and (\ref{Control-2}),  we return to the previous estimate of the term $I_{2,1}$ and write
\begin{equation*}
\begin{split}
I_{2,1}\leq &\,\,  C_1\,  \sum_{j=1}^{n}\sum_{k=1}^{n} t^\eta \, \int_{0}^{t} \left\| |\xi|\left( \widehat{h}_{\alpha_i} - \widehat{h}_2\right)(t-\tau,\cdot) \right\|_{L^\infty}\, \left\| u_{\alpha_j} u_{\alpha_k}(\tau,\cdot)  \right\|_{H^s}\, d\tau\\
\leq & \,\,  C_1  C(1+T^{\kappa+1})|2-\alpha_i| \left(  t^\eta \int_{0}^{t} (t-\tau)^{-\kappa} d\tau\right)   \times  \left( \sum_{j=1}^{n}\sum_{k=1}^{n} \left( \sup_{0\leq \tau\leq T}\| u_{\alpha_j} u_{\alpha_k}(\tau,\cdot)\|_{H^s}\right)\right)\\
\leq &\,\, C_1  C(1+T^{\kappa+1})|2-\alpha_i|\,  (t^{\eta+1-\kappa}) \times ( n^2\, C_2)\\
\leq &\,\, n^2\, C_1\, C_2 \, C(1+T^{\kappa+1})\,T^{\eta+1-\kappa}\, |2-\alpha_i|\\
=:& \,\,  \mathfrak{C}_3 (1+T^{\kappa+1})\,T^{\eta+1-\kappa}\, |2-\alpha_i|,\\
\end{split}
\end{equation*}
where, since $\kappa<1$ it follows that  $\eta+1-\kappa>0$. Finally, since $\ds{|2-\alpha_i| \leq \max_{i=1,\cdots,n}F_i(|2-\alpha_i|)}$, we obtain
\begin{equation}\label{I21}
I_{2,1}\leq \mathfrak{C}_3 (1+T^{\kappa+1})\,T^{\eta+1-\kappa} \, \max_{i=1,\cdots,n}F_i(|2-\alpha_i|).
\end{equation} 

\medskip

To control the second term $I_{2,2}$, we use again the definition of the operator $Q_{i,j,k}$ given in (\ref{Term-Q}) and the estimate $\widehat{q}_{i,j,k}(\xi) \leq C_1|\xi|$. Moreover, we apply well-known properties of the classical heat kernel $h_2$.   As a result, we obtain
\begin{equation*}
\begin{split}
I_{2,2}\leq &\, \sum_{j=1}^{n}\sum_{k=1}^{n} t^\eta \int_{0}^{t} \left\| h_2(t-\tau,\cdot)\ast Q_{i,j,k}\Big (u_{\alpha_j}(u_{\alpha_k}-u_{2,k}) + (u_{\alpha_j}-u_{2,j})u_{2,k}\Big)(\tau,\cdot) \right\|_{H^s} d \tau \\
\leq &\,C_1\, \sum_{j=1}^{n}\sum_{k=1}^{n} t^\eta \int_{0}^{t} \left\| \,|\xi|\,  \widehat{h}_2(t-\tau,\cdot) \, \mathcal{F}\Big (u_{\alpha_j}(u_{\alpha_k}-u_{2,k}) + (u_{\alpha_j}-u_{2,j})u_{2,k}\Big)(\tau,\cdot) (1+|\xi|^2)^{\frac{s}{2}}  \right\|_{L^2}\\
\leq &\, C_1\, \sum_{j=1}^{n}\sum_{k=1}^{n} t^\eta \, \int_{0}^{t} \left\| \,|\xi|\,  \widehat{h}_2(t-\tau,\cdot) \right\|_{L^\infty}\, \left\| \Big (u_{\alpha_j}(u_{\alpha_k}-u_{2,k}) + (u_{\alpha_j}-u_{2,j})u_{2,k}\Big)(\tau,\cdot) \right\|_{H^s} d \tau\\
\leq &\, C_1\, \sum_{j=1}^{n}\sum_{k=1}^{n} t^\eta \, C\,  \int_{0}^{t}  (t-\tau)^{-1/2} \, \left\| \Big (u_{\alpha_j}(u_{\alpha_k}-u_{2,k}) + (u_{\alpha_j}-u_{2,j})u_{2,k}\Big)(\tau,\cdot) \right\|_{H^s} d \tau.
\end{split}
\end{equation*} 
To estimate the second expression in the last integral, using product laws in Sobolev spaces and the uniform control (\ref{Control-Sol-2}), we can write, for every $0 < \tau \leq T$:
\begin{equation*}
\begin{split}
&\, \left\| \Big (u_{\alpha_j}(u_{\alpha_k}-u_{2,k}) + (u_{\alpha_j}-u_{2,j})u_{2,k}\Big)(\tau,\cdot) \right\|_{H^s}\\
 \leq & c_s\,\| u_{\alpha_j}(\tau,\cdot)\|_{H^s}\| (u_{\alpha_k}-u_{2,k})(\tau,\cdot)\|_{H^s} + c_s\, \| (u_{\alpha_j}-u_{2,j})(\tau,\cdot)\|_{H^s}\| u_{2,k}(\tau,\cdot)\|_{H^s} \\
 \leq &\, c_s(c\, M_{\delta}) \, \| (u_{\alpha_k}-u_{2,k})(\tau,\cdot)\|_{H^s} +c_s(c\, M_\delta)   \| (u_{\alpha_j}-u_{2,j})(\tau,\cdot)\|_{H^s}.
\end{split}
\end{equation*}
We then write, for $0\leq t \leq T$:
\begin{equation*}
\begin{split}
I_{2,2}\leq &\, C\,C_1\,c_s( c\, M_\delta)   \sum_{j=1}^{n}\sum_{k=1}^{n} t^\eta \int_{0}^{t}(t-\tau)^{-1/2}\, \Big( \|(u_{\alpha_k}-u_{2,k})(\tau,\cdot)\|_{H^s} +   \| (u_{\alpha_j}-u_{2,j})(\tau,\cdot)\|_{H^s} \Big) d \tau \\
=&\, C\,C_1\,c_s( c\, M_\delta)   n\, \sum_{j=1}^{n} t^\eta \int_{0}^{t}(t-\tau)^{-1/2}\, \|(u_{\alpha_j}-u_{2,j})(\tau,\cdot)\|_{H^s}  d \tau\\
&\, +C\,C_1\,c_s( c\, M_\delta)   n\, \sum_{k=1}^{n} t^\eta \int_{0}^{t}(t-\tau)^{-1/2}\,  \|(u_{\alpha_k}-u_{2,k})(\tau,\cdot)\|_{H^s} d \tau\\
=&\, 2n\, C\,C_1\, c_s(c\, M_\delta)\, \sum_{i=1}^{n} t^\eta \int_{0}^{t} (t-\tau)^{-1/2}\|( u_{\alpha_i} - u_{2,i})(\tau,\cdot)\|_{H^s} d\tau.
\end{split}
\end{equation*}
Denoting $\mathfrak{C}_4:=2n\, C\,C_1\, c_s(c\, M_\delta)$, we obtain
\begin{equation}\label{I22}
\begin{split}
I_{2,2} \leq &\,  \mathfrak{C}_4 \, \sum_{i=1}^{n} t^\eta \int_{0}^{t} (t-\tau)^{-1/2}\|( u_{\alpha_i} - u_{2,i})(\tau,\cdot)\|_{H^s} d\tau\\ 
=&\, \mathfrak{C}_4\,  t^\eta \int_{0}^{t} (t-\tau)^{-1/2}\, \tau^{-\eta}\, \left( \sum_{i=1}^{n}\tau^\eta  \|(u_{\alpha_i}-u_{2,i})(\tau,\cdot)\|_{H^s}\right)\, d\tau \\
\leq &\, \mathfrak{C}_4\, \left(t^\eta \, \int_{0}^{t} (t-\tau)^{-1/2}\, \tau^{-\eta} d\tau \right)\left( \sup_{0\leq \tau\leq T} \sum_{i=1}^{n}\tau^\eta  \|(u_{\alpha_i}-u_{2,i})(\tau,\cdot)\|_{H^s} \right)  \\
\leq & \, \mathfrak{C}_4\, T^{1/2}\,  \left( \sup_{0\leq \tau\leq T} \sum_{i=1}^{n}\tau^\eta  \|(u_{\alpha_i}-u_{2,i})(\tau,\cdot)\|_{H^s} \right).
\end{split}
\end{equation} 
Having the estimates (\ref{I21}) and (\ref{I22}) at our disposal, we deduce from estimate (\ref{I2-Previous}) that
\begin{equation}\label{I_2}
I_2 \leq \mathfrak{C}_3 (1+T^{\kappa+1})\,T^{\eta+1-\kappa}\, \max_{i=1,\cdots,n} F_i(|2-\alpha_i|)+ \mathfrak{C}_4\, T^{1/2}\, \left( \sup_{0\leq t\leq T} \sum_{i=1}^{n} t^\eta  \|(u_{\alpha_i}-u_{2,i})(t,\cdot)\|_{H^s} \right).
\end{equation}

\emph{Term $I_3$}.  We write
\begin{equation}
\begin{split}
I_3 \leq &\, \sum_{j=1}^{n} t^\eta \left\| \int_{0}^{t} (h_{\alpha_j} - h_2)(t-\tau, \cdot) \ast L_{i,j}(u_{\alpha_j})(\tau, \cdot) , d\tau \right\|_{H^s} \\
&\,+ \sum_{j=1}^{n} t^\eta \left\| \int_{0}^{t} h_2(t-\tau, \cdot) \ast L_{i,j}(u_{\alpha_j} - u_{2,j})(\tau, \cdot) , d\tau \right\|_{H^s} \\
=: &\, I_{3,1} + I_{3,2},
\end{split}
\end{equation}
where we recall that the operator $L_{i,j}(\cdot)$ is defined in expression (\ref{Term-L}) by the symbol $\widehat{\ell_{i,j}}(\xi)$, which is bounded for any $\xi \neq 0$. Consequently, for a constant $c_{i,j} > 0$, one has $\|L_{i,j}(u)\|_{H^s} \leq c_{i,j}\|u\|_{H^s}$. We also denote $c := \ds{\max_{i,j=1, \ldots, n} c_{i,j}}$.  

\medskip

To control the term $I_{3,1}$, recall that by estimate (\ref{Key-Estimate-1}), we obtain
\[ \| (\widehat{h}_{\alpha_j} - \widehat{h}_2)(t-\tau,\cdot)\|_{L^\infty} \leq C(1+T^{\eta+1})\,| 2-\alpha_i |\, (t-\tau)^{-\eta}. \]
Additionally, from estimate (\ref{Control-Sol-2}) it follows that 
\[   \sum_{j=1}^{n} \sup_{0\leq \tau\leq T} \| u_{\alpha_j}(\tau,\cdot)\|_{H^s}  \leq n\, c M_\delta. \]
Then,  we obtain
\begin{equation*}
\begin{split}
I_{3,1}\leq &\, \sum_{j=1}^{n} t^\eta \int_{0}^{t}\| (\widehat{h}_{\alpha_j} - \widehat{h}_2)(t-\tau,\cdot)\|_{L^\infty}\, \| L_{i,j}(u_{\alpha_j})(\tau,\cdot)\|_{H^s} \, d \tau \\
\leq &\, c\,  \sum_{j=1}^{n} t^\eta \int_{0}^{t}\| (\widehat{h}_{\alpha_j} - \widehat{h}_2)(t-\tau,\cdot)\|_{L^\infty}\, \| u_{\alpha_j}(\tau,\cdot)\|_{H^s} \, d \tau \\
\leq &\, c\,C(1+T^{\eta+1})|2-\alpha_i|\, \left( t^\eta \int_{0}^{t}(t-\tau)^{-\eta} d \tau \right) \left( \sum_{j=1}^{n} \sup_{0\leq \tau\leq T} \| u_{\alpha_j}(\tau,\cdot)\|_{H^s}\right)  \\
\leq &\,c\,  C(1+T^{\eta+1})|2-\alpha_i|\, T\, ( n \, c M_\delta) =:   \mathfrak{C}_5(1+T^{\eta+1})T\, \max_{i=1,\cdots,n}F_{i}(|2-\alpha_i|).
\end{split}
\end{equation*}
Similarly, to estimate the term $I_{3,2}$, using again well-known properties of the heat kernel $h_2$, we write
\begin{equation*}
\begin{split}
I_{3,2} \leq &\, c\, \sum_{j=1}^{n} t^\eta \int_{0}^{t}\| \widehat{h}_2(t-\tau,\cdot)\|_{L^\infty}\, \| (u_{\alpha_j}-u_{2,j})(\tau,\cdot)\|_{H^s}\, d\tau \\
\leq &\, c\, C\, \left( t^\eta\,  \int_{0}^t d\tau  \right)\, \left( \sup_{0\leq \tau \leq T} \sum_{j=1}^{n} \| (u_{\alpha_j}-u_{2,j})(\tau,\cdot)\|_{H^s}\right) \\ 
\leq &\, \mathfrak{C}_6 \, T^{\eta+1} \, \left(  \sup_{0\leq \tau\leq T} \sum_{j=1}^{n} \tau^\eta \| (u_{\alpha_j}-u_{2,j})(\tau,\cdot)\|_{H^s} \right).
\end{split}
\end{equation*}
Consequently, the term $I_3$ verifies:
\begin{equation}\label{I_3}
I_3 \leq  \mathfrak{C}_5(1+T^{\eta+1})T\, \max_{i=1,\cdots,n}F_{i}(|2-\alpha_i|) + \mathfrak{C}_6 \, T^{\eta+1} \, \left(  \sup_{0\leq \tau\leq T} \sum_{i=1}^{n} \tau^\eta \| (u_{\alpha_i}-u_{2,i})(\tau,\cdot)\|_{H^s} \right).
\end{equation} 

\medskip

Having estimated the terms $I_1$, $I_2$ and $I_3$ in (\ref{I1}), (\ref{I_2}), and (\ref{I_3}), respectively, we now return to inequality (\ref{Ineq1}) to obtain:
\begin{equation*}
\begin{split}
&\,\, t^\eta\| u_{\alpha_i}(t,\cdot)-u_{2,i}(t,\cdot)\|_{H^s} \\
\leq &\,\,   \mathfrak{C}_{1}\left(1+T^{\eta+1}\right)\,
\max_{i=1,\cdots,n}F_i(|2-\alpha_i|) \\ 
&\,\, + \mathfrak{C}_3 (1+T^{\kappa+1})\,T^{\eta+1-\kappa}\, \max_{i=1,\cdots,n} F_i(|2-\alpha_i|)+ \mathfrak{C}_4\, T^{1/2}\, \left( \sup_{0\leq \tau\leq T} \sum_{i=1}^{n}\tau^\eta  \|(u_{\alpha_i}-u_{2,i})(\tau,\cdot)\|_{H^s} \right)\\
&\,\, \mathfrak{C}_5(1+T^{\eta+1})T\, \max_{i=1,\cdots,n}F_{i}(|2-\alpha_i|) + \mathfrak{C}_6 \, T^{\eta+1} \, \left(  \sup_{0\leq \tau\leq T} \sum_{i=1}^{n} \tau^\eta \| (u_{\alpha_i}-u_{2,i})(\tau,\cdot)\|_{H^s} \right).
\end{split} 
\end{equation*}
Then, rearranging the terms, we write:
\begin{equation*}
\begin{split}
&\,\, \left(\sup_{0\leq t \leq T} \sum_{i=1}^{n}  t^\eta\| u_{\alpha_i}(t,\cdot)-u_{2,i}(t,\cdot)\|_{H^s} \right)  \\
\leq &\,\,   n\times \max( \mathfrak{C}_{1}, \mathfrak{C}_3,  \mathfrak{C}_5) \Big((1+T^{\eta+1})+ (1+T^{\kappa+1})\,T^{\eta+1-\kappa}+ (1+T^{\eta+1})T\Big)\,
\max_{i=1,\cdots,n}F_i(|2-\alpha_i|) \\ 
&\,\, + n \times \max(\mathfrak{C}_4,\mathfrak{C}_6)\left( T^{1/2}+T^{\eta+1}\right) \left(  \sup_{0\leq \tau\leq T} \sum_{i=1}^{n} \tau^\eta \| (u_{\alpha_i}-u_{2,i})(\tau,\cdot)\|_{H^s} \right).
\end{split} 
\end{equation*}

In the first term on the right-hand side, to obtain a simpler expression involving the time $T$, recall that $T<\min(1,T_0)$ and  $\eta<\sigma$.  Hence:
\[ \Big((1+T^{\eta+1})+ (1+T^{\kappa+1})\,T^{\eta+1-\kappa}+ (1+T^{\eta+1})T\Big) \leq 3(1+T^{\eta+1}).\]
In addition, we define the constant  ${\bf C}:= 6\,n\times \max( \mathfrak{C}_{1},\mathfrak{C}_3,\mathfrak{C}_5)$.

\medskip

On the other hand, in the second term on the right-hand side, we set  $T$ sufficiently small such that: 
\[ n \times  \max(\mathfrak{C}_4,\mathfrak{C}_6)\left( T^{1/2}+T^{\eta+1}\right) \leq \frac{1}{2}. \]

With these estimates, it follows that: 
\begin{equation}\label{Main-Esim-T}
\left(\sup_{0\leq t \leq T} \sum_{i=1}^{n}  t^\eta\| u_{\alpha_i}(t,\cdot)-u_{2,i}(t,\cdot)\|_{H^s} \right)  \leq {\bf C} (1+T^{\eta+1})\,  \max_{i=1,\cdots,n}F_i(|2-\alpha_i|). 
\end{equation}

\subsubsection{Step 2}\label{Sec:Step-2}  For the reader's convenience, we reiterate that each component of the solution $(u_{\alpha_1}, \ldots, u_{\alpha_n})$ is defined over the interval of time $[0,T_{\alpha}]$, while the limiting solution $(u_{2,1},\cdots, u_{2,n})$ exists on  $[0,T_2]$. Moreover, by our assumption (\ref{Assumption-Times}), we have  $T_{\alpha}\leq T_2$. 

\medskip

To complete the proof of Theorem \ref{Th-Main}, we extend the estimate (\ref{Main-Esim-T}) to the interval of time $[0,T_2]$ as follows. First, each component $u_{\alpha_i}$ is extended by zero to the whole interval $[0,T_2]$. We then define $\tilde{u}_{\alpha_i}:[0,T_2]\times \R^d \to  \R$ as:
\begin{equation*}
\tilde{u}_{\alpha_i}(t,\cdot)=\begin{cases}\vspace{1mm}
u_{\alpha_i}(t,\cdot), & \quad 0\leq t \leq T_{\alpha}, \\
0, &\quad T_{\alpha}<t\leq T_2.
\end{cases}
\end{equation*}
For the time $T>0$ fixed above, the estimate (\ref{Main-Esim-T}) holds for the extended solution $(\tilde{u}_{\alpha_1}, \cdots, \tilde{u}_{\alpha_n})$. Finally, we iteratively apply this argument up to the time $T_2$, obtaining the desired estimate (\ref{Conv-Rate-Sol}). Theorem \ref{Th-Main} is now proven.

\section{Proof of Corollary \ref{Corollary}}
This corollary directly follows from estimate (\ref{Conv-Rate-Sol-Epsilon}) and the well-known interpolation and Gagliardo–Nirenberg inequalities. Indeed, in the case of the $L^p$-spaces, by the continuous embeddings $H^s(\R^d)\subset L^2(\R^d)$ and $H^s(\R^d)\subset L^\infty(\R^d)$ (since $s>d/2$), for $2\leq p \leq +\infty$ and $\theta=\frac{2}{p}\in [0,1]$, we can write 
\begin{equation*}
\begin{split}
\| u_{\alpha_i}(t,\cdot)-u_{2,i}(t,\cdot)\|_{L^p} \leq  &\, C \| u_{\alpha_i}(t,\cdot)-u_{2,i}(t,\cdot)\|^{\theta}_{L^2}\, \| u_{\alpha_i}(t,\cdot)-u_{2,i}(t,\cdot)\|^{1-\theta}_{L^\infty} \\
\leq &\, C \| u_{\alpha_i}(t,\cdot)-u_{2,i}(t,\cdot)\|_{H^s}.  
\end{split} 
\end{equation*} 
Similarly, in the case of the $\dot{W}^{\sigma,p}$-spaces, for $0<\sigma<s$ and $2\leq p<+\infty$ such that $\frac{1}{p}=\frac{\theta}{q}+\frac{1-\theta}{2}$, with $2\leq q \leq +\infty$ and  $\theta = 1-\frac{\sigma}{s}\in (0,1)$, we obtain
\begin{equation*}
\begin{split}
\| u_{\alpha_i}(t,\cdot)-u_{2,i}(t,\cdot)\|_{\dot{W}^{\sigma,p}} \leq &\, C \| u_{\alpha_i}(t,\cdot)-u_{2,i}(t,\cdot)\|^{\theta}_{L^{q}} \, \| u_{\alpha_i}(t,\cdot)-u_{2,i}(t,\cdot)\|^{1-\theta}_{\dot{H}^{s}}\\
\leq&\, C \| u_{\alpha_i}(t,\cdot)-u_{2,i}(t,\cdot)\|_{H^s}.
 \end{split}
\end{equation*}

%%%%%%%%%%%%%%%%%%%%%%%%%%%%%%%%%%%%%%%%%%%%%%


\begin{thebibliography}{99}
 	\bibitem{Bessaih} H. Bessaih and B. Ferrario. The regularized 3D Boussinesq equations with fractional Laplacian and no diffusion. Journal of Differential Equations. Volume 262, Issue 3:  1822-1849 (2017). 
 	%%%%%%%%%%%%%%%%%%%
 	\bibitem{Bezerra} F. Bezerra, A.N. Caravalho and M.J.D. Nascimento, \emph{Fractional approximations of abstract semilinear parabolic problems}, Discrete and Continuous Dynamical Systems - B,  25(11) (2020).
 	%%%%%%%%%%%%%%%%%%%
 	\bibitem{Biccari} U. Biccari and  V. Hern\'andez-Santamar\'ia, \emph{The poisson equation from non-local to local}. Electronic Journal
 	of Differential Equations, Vol. 2018 No. 145, pp. 1–13 (2018).
 	%%%%%%%%%%%%%%%%%%%
 	 	\bibitem{Bieganowski} B.  Bieganowski, S. Secchi. \emph{Non-local to local transition for ground states of fractional Schr\"odinger equations on $\R^N$}. J. Fixed Point Theory Appl. 22, 76 (2020).
 	%%%%%%%%%%%%%%%%%%%
 	\bibitem{Brandolese} L. Brandolese and S.  Monniaux. \emph{Well-posedness for the Boussinesq system in critical spaces via maximal regularity}. Annales de l'Institut Fourier, Tome 73, no. 1, pp. 1-20 (2023).
 	%%%%%%%%%%%%%%%%%%%
 	\bibitem{Cannon} J.R. Cannon and E. Dibenedetto. \emph{The initial value problem for the Boussinesq equations with data in $L^p$}. Approximation methods for Navier-Stokes
 	problems (Proc. Sympos., Univ. Paderborn, Paderborn, 1979), Lecture
 	Notes in Math. 771, Springer, Berlin, 1980, pp. 129–144. 
 	%%%%%%%%%%%%%%%%%%%
 		\bibitem{Chamorro} D. Chamorro and S. Menozzi, \emph{Non linear singular drifts and fractional operators}. Partial Differ. Equ. Appl. 5, 33 (2024).
 	%%%%%%%%%%%%%%%%%%%
 	\bibitem{Dlotko} T. Dlotko. Navier–Stokes Equation and its Fractional Approximations. Applied Mathematics \& Optimization 77(1) (2018).
 	%%%%%%%%%%%%%%%%%%%
 	 	\bibitem{Escudero} C. Escudero, \emph{The fractional Keller–Segel model}.  Nonlinearity 19,  2909–2918 (2006).
 	%%%%%%%%%%%%%%%%%%%
 	 \bibitem{Fernandez-Salort} J. Fern\'andez Bonder and A.M. Salort. \emph{Stability of solutions for nonlocal problems}, Nonlinear Analysis, Volume 200, 112080, (2020).  
 	%%%%%%%%%%%%%%%%%%%
 	\bibitem{Funaki} T. Funaki, D. Surgailis and W. A. Woyczynski, \emph{Gibbs-Cox random fields and Burgers turbulence}. Ann. Appl. Prob. 5, 701-735 (1995).
 	%%%%%%%%%%%%%%%%%%%
 	\bibitem{Holst} M. Holst, E.M. Lunasin and G. Tsogtgerel, \emph{Analysis of a General Family of Regularized Navier-Stokes
 	and MHD Models}.  Journal of Nonlinear Science 20(5) (2009).
 	%%%%%%%%%%%%%%%%%%%
 	\bibitem{Keller-Segel} E. F. Keller and L. A. Segel, \emph{Initiation of slime mold aggregation viewed as an instability, J.
 	Theor. Biol}. 26, pp. 399–415 (1970).
 	%%%%%%%%%%%%%%%%%%%
 	 	\bibitem{Jarrin_L2024} Jarrín, O., Loachamín, G. From non-local to local Navier-Stokes equations. J. Appl. Math. Optim,  Volume 89, article number 61,  (2024).
 	%%%%%%%%%%%%%%%%%%%
 	\bibitem{Jourdain} B. Jourdain, S. Méléard and W. Woyczy\'nski, \emph{ A probabilistic approach for nonlinear equations involving the fractional Laplacian and singular operator}.  Potential Analysis 23 , 55–81 (2005).
 	%%%%%%%%%%%%%%%%%%%
 	\bibitem{PGLibro} P.G. Lemari\'e-Rieusset, \emph{The Navier-Stokes Problem in the 21st Century} Chapman \& Hall/CRC, (2016).
 	%%%%%%%%%%%%%%%%%%%
 	\bibitem{PG1} P.G. Lemarié-Rieusset,
 \emph{Sobolev multipliers, maximal functions and parabolic equations with a quadratic nonlinearity},
 	Journal of Functional Analysis,Volume 274, Issue 3, Pages 659-694 (2018).
 	%%%%%%%%%%%%%%%%%%% 
 	\bibitem{Metzler} R. Metzler, J. Klafter, \emph{The random walks guide to anomalous diffusion: A fractional dynamics approach}. Phys. Rep. 339,  1–77 (2000). 
 	%%%%%%%%%%%%%%%%%%%
 	  \bibitem{Pedlosky} J. Pedlosky, \emph{Geophysical fluid dynamics}, Springer, (1987).
 	%%%%%%%%%%%%%%%%%%%
 	\bibitem{Rozkosz} A. Rozkosz, L. S\l{}omi\'nski, \emph{Stability of solutions of semilinear evolution equations with integro-differential operators}.  arXiv.2407.06666  (2024). 
 	%%%%%%%%%%%%%%%%%%%
 \bibitem{Shercliff}	J. A. Shercliff, \emph{A Textbook of Magneto-hydrodynamics}, Pergamon Press, Oxford, (1965)
 	%%%%%%%%%%%%%%%%%%%
 	\bibitem{Wu} G. Wu, X. Zheng, On the well-posedness for Keller–Segel system with fractional diffusion, Math.
 	Methods Appl. Sci. 34(14),  1739–1750 (2011). 
 	%%%%%%%%%%%%%%%%%%%
 	\bibitem{Xu} X, Xu, \emph{Global regularity of solutions of 2D Boussinesq equations with fractional diffusion}. Nonlinear Analysis: Theory, Methods \& Applications
 	Volume 72, Issue 2: 677-681 (2010). 
 	%%%%%%%%%%%%%%%%%%%
 	\bibitem{Xu-Ca} J. Xu, T. Caraballo and J. Valero, \emph{On the Limit of Solutions for a Reaction–Diffusion Equation Containing Fractional Laplacians}. Appl Math Optim 89, 22 (2024).
 	%%%%%%%%%%%%%%%%%%%
 	\bibitem{Chamorro-Yangari} D. Chamorro and M. Yangari, \emph{Some existence and regularity results for a non-local transport-diﬀusion equation with fractional derivatives in time and space}. arXiv:2203.13101v1 (2022).
 	%%%%%%%%%%%%%%%%%%%
 	\bibitem{Yu-Zhai} X. Yu and Z. Zhai. \emph{Well-posedness for fractional Navier–Stokes equations in the largest critical spaces $\dot{B}^{-2(\beta-1)}_{\infty,\infty}(\R^n)$}.  Mathematical Methods in Applied Science, 35, pp. 676–683 (2012).
 	%%%%%%%%%%%%%%%%%%%
% 	\bibitem{Behzadan-Holst} A. Behzadan and M. Holst. \emph{Multiplication in Sobolev spaces, revisited}.  Arkiv för Matematik, 59, pp. 275-306 (2021).
 	%%%%%%%%%%%%%%%%%%%
 	\bibitem{Zaslavsky} G. M. Zaslavsky and S. S. Abdullaev, \emph{Scaling properties and anomalous transport of particles inside the
 	stochastic layer}. Phys. Rev. E 51, No. 5 3901-3910 (1995).	
    %%%%%%%%%%%%%%%%%%%
    \bibitem{Zhai} Z. Zhai, \emph{Global well-posedness for nonlocal fractional Keller–Segel systems in critical Besov
    spaces}. Nonlinear Anal. 72, 3173–3189 (2010). 
 \end{thebibliography}
\end{document}